\documentclass[hidelinks,onefignum,onetabnum]{siamart220329}

\usepackage{lipsum}
\usepackage{amsfonts}
\usepackage{graphicx}
\usepackage{epstopdf}
\usepackage{algorithmic}
\ifpdf
  \DeclareGraphicsExtensions{.eps,.pdf,.png,.jpg}
\else
  \DeclareGraphicsExtensions{.eps}
\fi

\newsiamremark{remark}{Remark}
\newsiamremark{hypothesis}{Hypothesis}
\crefname{hypothesis}{Hypothesis}{Hypotheses}
\newsiamthm{claim}{Claim}
\newsiamthm{conjecture}{Conjecture}

\def\Re{\hbox{\rm Re\kern .8pt}}\def\Im{\hbox{\rm Im\kern .8pt}}
\def\Imm{\hbox{\scriptsize\rm Im\kern .8pt}}
\def\Ree{\hbox{\scriptsize\rm Re\kern .8pt}}

\def\Ileft{I_{\kern .5pt\rm\footnotesize left}(x)}
\def\Iright{I_{\kern .5pt\rm\footnotesize right}(x)}
\def\Imid{I_{\kern .5pt\rm\footnotesize mid}(x)}
\def\Eleft{E_{\kern .5pt\rm\footnotesize left}(x)}
\def\Eright{E_{\kern .5pt\rm\footnotesize right}(x)}
\def\Esides{E_{\kern .5pt\rm\footnotesize sides}(x)}
\def\Eresidues{E_{\kern .5pt\rm\footnotesize residues}(x)}

\headers{Resolution of Singularities by Rational Functions}{A. Herremans, D. Huybrechs, and L. N. Trefethen}

\title{Resolution of Singularities by Rational Functions}
\author{Astrid Herremans \thanks{Department of Computer Science, KU Leuven, 3001 Leuven, Belgium
  (\email{astrid.herremans@kuleuven.be}, \email{daan.huybrechs@kuleuven.be}).}
\and Daan Huybrechs \footnotemark[1]
\and Lloyd N. Trefethen \thanks{Mathematical Institute, University of Oxford, Oxford OX2 6GG, UK
  (\email{trefethen@maths.ox.ac.uk}).}}

\usepackage{amsopn}

\newcommand{\norm}[1]{\left\lVert#1\right\rVert}

\ifpdf
\hypersetup{
  pdftitle={Resolution of Singularities by Rational Functions},
}
\fi

\begin{document}

\maketitle
% REQUIRED
\begin{abstract}
Results on the rational approximation of functions containing singularities are presented. We build further on the ``lightning method'', recently proposed by Trefethen and collaborators, based on exponentially clustering poles close to the singularities. Our results are obtained by augmenting the lightning approximation set with either a low-degree polynomial basis or partial fractions with poles clustering towards infinity, in order to obtain a robust approximation of the smooth behaviour of the function. This leads to a significant increase in the achievable accuracy as well as the convergence rate of the numerical scheme. For the approximation of $x^\alpha$ on $[0,1]$, the optimal convergence rate as shown by Stahl in 1993 is now achieved simply by least-squares fitting. 
\end{abstract}

% REQUIRED
\begin{keywords}
rational functions, approximation theory, complex analysis, least-squares, ill-conditioning
\end{keywords}

% REQUIRED
\begin{MSCcodes}
41A20, 65E05, 65F20
\end{MSCcodes}

\section{Introduction}
Recently, ``lightning approximations'' \cite{gopalSolvingLaplaceProblems2019} for analytic functions with branch point singularities have been proposed. The function is approximated by a rational function with preassigned poles, resulting in root-exponential convergence. Whereas computing rational approximations is a nonlinear problem, fixing the poles in advance linearizes the problem so that the approximant can be found by a matrix least-squares computation. This numerical scheme led to the development of ``lightning solvers'' for the 2D Laplace and Helmholtz equations \cite{fluids5040227, gopalNewLaplaceHelmholtz2019, Trefethenconformal} and more recently for the biharmonic equation, specifically for 2D Stokes flow \cite{brubecklightningStokesSolver2022}. The solutions to these PDEs exhibit singular behaviour near corners of the computational domain. The singularities are effectively resolved by approximating them with analytic functions.

In general, one can represent a rational approximation $r$ to a function $f$ on an approximation domain $E$ as follows:
\begin{equation}
    f(z) \; \approx \; r(z) = \frac{p(z)}{q(z)} = \sum_{j=1}^{N_1} \frac{a_j}{z-p_j} + \sum_{j=0}^{N_2} b_j z^j,
    \label{eq:ratapprox}
\end{equation}
assuming that the finite poles of $r$ are simple. The rational function $r$ has $N_1$ finite poles ${p_j}$ and $N_2$ poles at infinity, which can be described by a polynomial $b(z) = \sum_{j=0}^{N_2} b_j z^j$ of degree $N_2$. This is made explicit by expressing $r$ using partial fractions. The total degree of $r$ is $N = N_1 + N_2$. Approximation with rational functions overcomes two important problems of polynomial approximation. On the one hand, rational functions are able to converge root-exponentially for functions that contain singularities on $E$, in contrast to algebraic convergence for polynomial approximation. Secondly, rational functions are suitable for approximating functions on unbounded domains. 

The lightning method \cite{gopalSolvingLaplaceProblems2019} aims at efficiently approximating functions with branch point singularities at known locations $\{z_i\}$. To that end, the rational approximation problem~\eqref{eq:ratapprox} is linearized by fixing the poles of $r$ in advance. Specifically, one fixes a sequence of finite, simple poles exponentially clustered near each singularity, in order to approximate the local singular behaviour. The set of $N_1$ finite poles $\{p_j\}$ can therefore be partitioned in subsets, each related to a singularity $\{z_i\}$ of $f$. Optionally, a polynomial $b(z)$ is added to the approximation set as well. Thereafter the problem is oversampled and a least-squares system is solved to find the coefficients ${a_j}$ and ${b_j}$ of the discrete best approximation. One oversamples at least linearly in the number of poles, with sample points that are exponentially clustered towards the singularities.

In-depth research has been done on the optimal distribution and convergence behaviour of the finite ``lightning poles''. Root-exponential convergence of the rational approximant is guaranteed for any exponential clustering distribution, provided that it scales with $n^{-1/2}$ as $n \to \infty$ (with $n$ the number of poles clustering towards a singularity) \cite{gopalSolvingLaplaceProblems2019}. The distance of the closest pole to the singularity therefore satisfies $\mathcal{O}(\exp(-\sigma \sqrt{n}\kern1pt))$. The parameter $\sigma$ controls the spacing between the poles and therefore the rate at which they approach the singularity. Further research in \cite{trefethenExponentialNodeClustering2021} shows that ``tapered'' rather than ``uniform'' exponential clustering doubles the rate of convergence. We refer to an approximation using only finite, exponentially clustered poles as a ``lightning approximation''. Except where explicitly stated otherwise, the lightning poles we discuss are distributed in a tapered fashion.

In contrast, this paper focuses on the smooth part of the approximation problem, which has not previously been extensively studied. We investigate the influence of adding $N_2$ poles at infinity, i.e.\ of adding a polynomial of degree $N_2$ to the approximation set. An approximation using $N_1$ lightning poles and $N_2 = \mathcal{O}(\sqrt{N_1}\kern1pt)$ poles at infinity is referred to as a ``lightning + polynomial approximation''.

\subsection{Main results}
For thirty years it has been known that minimax rational approximations to $x^\alpha$ on $[0,1]$ converge at the rate $\mathcal{O}(\exp(-2 \pi \sqrt{\alpha N}\kern1pt))$ \cite{stahlBestUniformRational1993}. Here we show that this optimal rate can be achieved by a lightning + polynomial approximation, i.e.\ by preassigning poles and solving a least-squares problem. We derive this for $\alpha = 1/2$ (\cref{sec:2}) and show it numerically for other values of $\alpha$ (\cref{sec:3}).

We show that inclusion of this low-degree polynomial term is also crucial to numerical stability. With it, coefficient vectors are of modest size and least-squares fits can be quickly computed to close to machine precision.  Without it, coefficient vectors grow exponentially and the convergence stagnates at a much lower accuracy (\cref{sec:4}).

For optimal convergence rates in this model problem, we show that the clustering parameter should be $\sigma = 2\pi/\sqrt{\alpha}$ using tapered poles; for $\alpha = 1/2$ this is $2 \pi \sqrt{2} \approx 8.9$.  By analyzing model problems related to PDEs in regions with corners, however, we explain why the smaller value $\sigma = 4$ has proved effective in many cases (\cref{sec:5}).

Instead of a polynomial term of degree $\mathcal{O}(\sqrt{N}\kern1pt)$, i.e.~poles at infinity, one can achieve the minimax convergence rates with $\mathcal{O}(\sqrt{N}\kern1pt)$ additional finite poles ``tapered at infinity''.  \Cref{sec:2} also presents results on the asymptotics of such big poles for the approximation of $\sqrt{x}$ on $[0,1]$, but we argue that simple polynomial terms are probably better for applications.

\section{\boldmath Approximation of $\sqrt{x}$ on $[0, 1]$}
\label{sec:2}

The rational approximation of $\sqrt{x}$ is an example of historical interest \cite{newman1964,stahl1994rational,Vya74}, and it also serves as a reference case for other branch point singularities. Note that the approximation of $\sqrt{x}$ on $[0,1]$ is equivalent to the approximation of $\left| u \right|$ on $[-1,1]$ by the substitution $x = u^2$.

\subsection{Comparison with previous results}
The most recent results regarding the rational approximation of $\sqrt{x}$ on $[0,1]$ with preassigned exponentially clustered poles on $[-C,0]$ can be found in \cite{trefethenExponentialNodeClustering2021}, where a comparison is made between uniform exponentially clustered poles, meaning in the notation of~\eqref{eq:ratapprox}:
\begin{equation}
    p_j = -C\exp(-\sigma j / \sqrt{N_1} \kern1pt ), \qquad 0 \leq j \leq N_1-1,
    \label{eq:uniformpoles}
\end{equation}
and tapered exponentially clustered poles:
\begin{equation}
    p_j = -C\exp(-\sigma (\sqrt{N_1} - \sqrt{j} \kern1pt )), \qquad 1 \leq j \leq N_1.
    \label{eq:taperedpoles}
\end{equation}
For both clusterings, the smallest pole is of size $\mathcal{O}(\exp{(-\sigma \sqrt{N_1} \kern1pt )})$. A constant term is also included in the approximation set.

In \cite{trefethenExponentialNodeClustering2021}, $\sigma = \pi$ and $\sigma = \sqrt{2}\pi$ are used for the uniform and tapered clusterings, respectively. Here, we scale these poles by setting $C=2$ in \cref{eq:uniformpoles} and \cref{eq:taperedpoles}. These results are compared with the lightning + polynomial approximation, which is constructed by including a polynomial of degree $N_2 = \texttt{ceil}(1.3 \sqrt{N_1} \kern1pt)$. The precise value of $1.3$ is unimportant and could be increased without much affecting the results. In addition, the parameter $\sigma$ is increased by a factor of 2 in comparison with \cite{trefethenExponentialNodeClustering2021}, resulting in sparser uniform and tapered clusterings with $\sigma = 2\pi$ and $2\sqrt{2}\pi$, respectively. Figure \ref{fig:sqrtxcomparison} compares the errors as a function of the total degree $N$ of the different rational approximations. Note that for the lightning approximations, $N = N_1$, while for the lightning + polynomial approximations, one has $N = N_1 + N_2$. The figure also compares these approximation methods with rational best approximations, based on solving a nonlinear approximation problem with free poles. This problem was studied by Vyacheslavov, who proved that the best rational approximation to $\sqrt{x}$ on $[0,1]$ converges as $\mathcal{O}(\exp(-\pi \sqrt{2N} \kern1pt))$ \cite{Vya74}.

In Figure \ref{fig:sqrtxcomparison}, one can see that tapering results in a factor of 2 increase in the convergence rate for both the lightning and the lightning + polynomial approximations, as described in \cite{trefethenExponentialNodeClustering2021}. However, adding a polynomial term and increasing the clustering parameter $\sigma$ has an even bigger influence. For both the uniform and the tapered clusterings, the convergence rate increases by a factor of about $3$. This leads to the tapered lightning + polynomial approximation converging asymptotically at the same rate as the best rational approximation!

The system matrices obtained after sampling are heavily ill-conditioned, yet accurate results are still obtained. This phenomenon is explained in~\cref{sec:4}. Note that the lightning + polynomial approximations converge to a level much closer to machine precision than the lightning approximations. This effect is related to the norm of the coefficient vector.
    
\begin{figure}
    \centering
    \includegraphics[width = .8\linewidth]{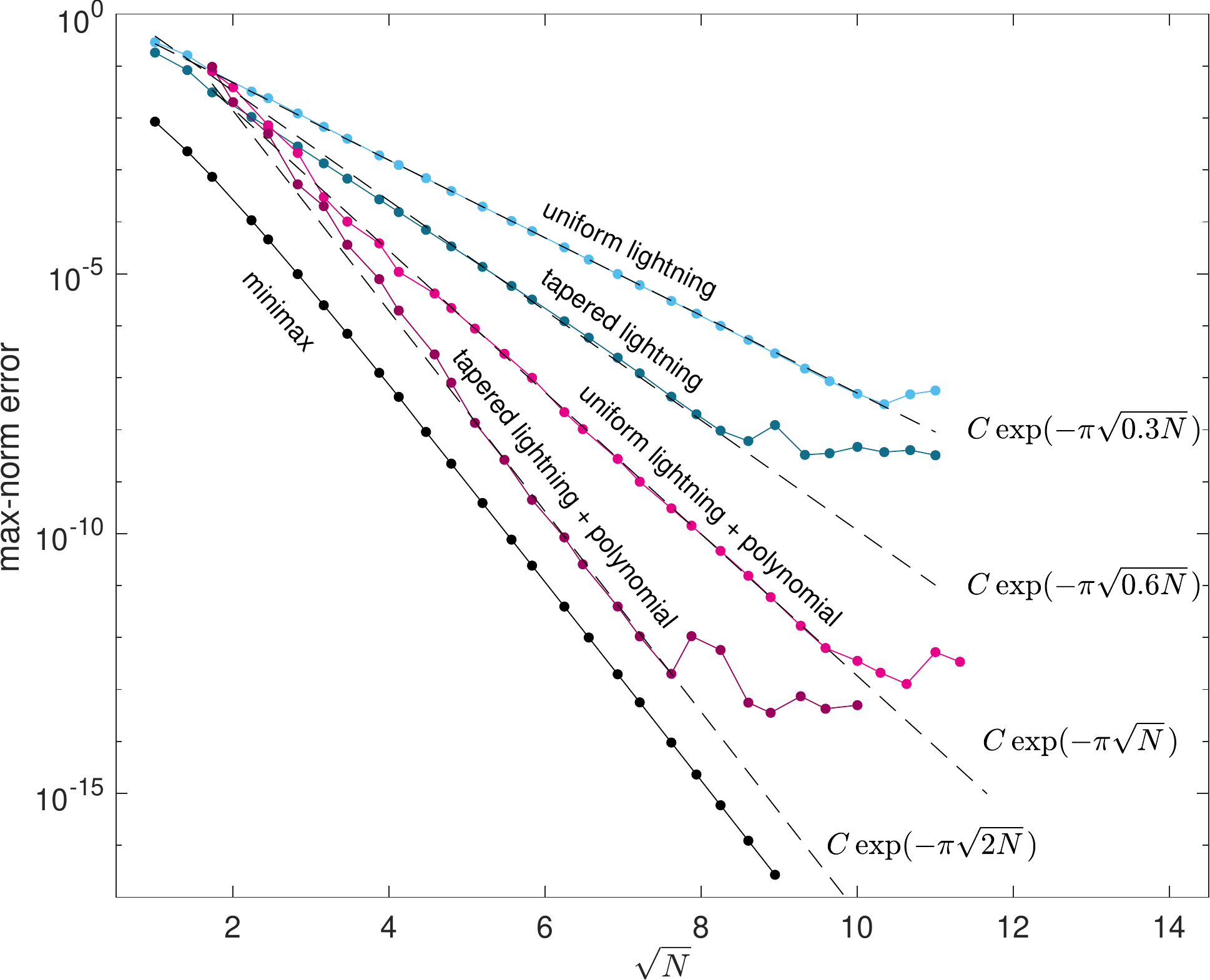} 
    \caption{Max-norm errors of the lightning approximations and rational minimax approximation for $\sqrt{x}$ on $[0,1]$. The convergence rates for the upper two curves are empirical.}
    \label{fig:sqrtxcomparison}
\end{figure}

\subsection{Construction of the lightning + polynomial approximation} \label{sec:22}
This section shows how a lightning + polynomial approximation of $\sqrt{x}$ on $[0,1]$ can be constructed that converges at the rate of the best rational approximation. In practice, accurate approximations are however found by numerically computing the discrete best approximation. \Cref{sec:4} shows how the existence of an accurate approximation allows one to understand the accuracy of the numerical method. A number of insights follow from this construction, which are used to examine more general approximation problems in \cref{sec:3,sec:5}.

The construction is inspired by the rational approximation described in \cite[Chapter 25]{trefethenApproximationTheoryApproximation2019}, which originates with \cite{stenger}. The approximation starts from the following integral representation of $\sqrt{x}$:
\begin{equation*}
    \sqrt{x} = \frac{2x}{\pi} \int^{\infty}_{0} \frac{1}{t^2+x} dt = \frac{2x}{\pi} \int_{-\infty}^{\infty} \frac{e^s}{e^{2s}+x} ds,
\end{equation*}
where $t = e^s$. One can approximate this integral using a quadrature rule, giving rise to a rational function of $x$ approximating $\sqrt{x}$. In \cite[Chapter 25]{trefethenApproximationTheoryApproximation2019} this is achieved by truncating the integral at $\pm T$:
\begin{equation}
    \sqrt{x} \approx \frac{2x}{\pi} \int_{-T}^{+T} \frac{e^s}{e^{2s}+x} ds,
    \label{eq:trunc}
\end{equation}
and thereafter using the trapezoidal rule to discretize the integral. We now adapt this construction to arrive at an approximation with tapered poles, by introducing the substitution $s + T = \sqrt{u}$:
\begin{equation}
     \sqrt{x} \approx \frac{2x}{\pi} \int_{0}^{4T^2} \frac{1}{2\sqrt{u}} \left( \frac{e^{\sqrt{u}-T}}{e^{2(\sqrt{u}-T)}+x} \right) \kern2pt du.
     \label{eq:inteq}
\end{equation}
Discretization using the trapezoidal rule in $N_t$ quadrature points with a stepsize $h$,
$$u = jh, \qquad 1 \leq j \leq N_t,$$ 
with $T = \sqrt{N_t h/4}$, then gives rise to the following rational approximation to $\sqrt{x}$:
\begin{equation}
    \sqrt{x} \approx r_t(x) = \frac{xh}{\pi} \sum_{j=1}^{N_t} \frac{1}{\sqrt{jh}} \left( \frac{e^{\sqrt{jh}-\sqrt{N_t h/4}}}{e^{2\sqrt{jh}-2\sqrt{N_t h/4}} + x} \right).
    \label{eq:fulltrap}
\end{equation}
The error of this rational approximation has two components: a truncation error, due to the truncation involved in \cref{eq:trunc}, and a discretization error, due to approximation of the truncated integral in \cref{eq:inteq} with the trapezoidal rule. The truncation error is of magnitude $\mathcal{O}(\exp(-T)) = \mathcal{O}(\exp(-\sqrt{N_t h/4}\kern1pt))$. In contrast, the discretization error increases as $h$ increases. Optimally, these two errors are balanced. From the following theorem, it follows that this is the case for a stepsize $h = 2\pi^2$.

\begin{theorem}
    The rational approximation \cref{eq:fulltrap} with $h = 2\pi^2$ converges to $\sqrt{x}$ with approximation error:
    $$\lvert r_t(x) - \sqrt{x} \kern1pt \rvert < 20 \kern1pt e^{-\sqrt{N_th/4}}\kern1pt = 20 \kern1pt e^{-\pi\sqrt{N_t/2}}$$
    as $N_t \to \infty$, uniformly for $x \in [0,1]$, assuming the bound introduced in \cref{conj:bound} holds.
    \label{thm:trap}
\end{theorem}
\begin{proof}
    See \cref{app1}, where an argument is given based on the representation of $r_t(x) -\sqrt{x}$ by a contour integral.
\end{proof}

The large poles of \cref{eq:fulltrap} can be approximated by a low-degree polynomial with exponential convergence. This results in a lightning + polynomial approximation close to the trapezoidal rule approximation, yet of essentially 4 times lower degree.

\begin{lemma}
     There exist coefficients $\{a_j\}_{j=1}^{N_1}$ with $N_1 = N_t/4$ and a polynomial $b(x)$ of degree $N_2 = \mathcal{O}(\sqrt{N_1} \kern1pt)$, for which the lightning + polynomial approximation $r(x)$ \cref{eq:ratapprox} having tapered lightning poles \cref{eq:taperedpoles} with $\sigma = 2\sqrt{h}$, satisfies:
    $$\lvert r_t(x) - r(x) \rvert = \mathcal{O}\left( e^{-\sqrt{N_t h/4}}\kern1pt \right)$$
    as $N_t \to \infty$, uniformly for $x \in [0,1]$. I.e., the error is of the same magnitude as the approximation error involved in \cref{thm:trap}.
    \label{lemma:lp}
\end{lemma}
\begin{proof}
Consider \cref{eq:fulltrap} in partial fractions form, while substituting $N_t=4N_1$:
\begin{equation*}
    r_t(x) = \sum_{j=1}^{4N_1} \frac{a_j}{x - p_j} + C.
\end{equation*}
A bit of calculation determines the poles, residues, and constant term:
\begin{align}
    \text{poles:} \qquad p_j &= -\exp(-2\sqrt{h}(\sqrt{N_1} - \sqrt{j} \kern1pt )), \qquad &1 \leq j \leq 4N_1, \label{eq:trapezoidal_poles}\\
    \text{residues:} \qquad a_j &= \frac{\sqrt{h}}{\pi} \kern2pt p_j \kern1pt \sqrt{\frac{\lvert p_j \rvert}{j}}, \qquad &1 \leq j \leq 4N_1, \nonumber \\
    \text{constant term:} \qquad C &= \frac{\sqrt{h}}{\pi} \kern2pt \sum_{j=1}^{4N_1} \kern2pt \sqrt{\frac{\lvert p_j \rvert}{j}}.\nonumber
\end{align}
The first $N_1$ poles of this rational approximation are exactly the tapered lightning poles \cref{eq:taperedpoles} with $\sigma = 2\sqrt{h}$. The remaining $3N_1$ poles are ``large'' in the sense that $\lvert p_j \rvert > 1$. These large poles and the constant term can be approximated by a polynomial $b(x)$ of degree $N_2$:
\begin{equation}
    b(x) = \sum_{j=0}^{N_2} b_j x^j \approx \sum_{j=N_1+1}^{4N_1} \frac{a_j}{x - p_j} + C.
    \label{eq:bN2}
\end{equation}
Since the right-hand side is analytic in a domain containing $[0,1]$ in its interior, the polynomial approximation may converge exponentially \cite[Chapter 8]{trefethenApproximationTheoryApproximation2019}. The rate of convergence depends on the Bernstein ellipse $E_\rho$ wherein the function is analytic, after transplanting the problem to the interval $[-1,1]$. This region is determined by the position of the closest pole, namely $p_{N_1+1} = -\exp(-2\sqrt{h}(\sqrt{N_1} - \sqrt{N_1+1}\kern1pt)) < -1$. A pole at $-1$ results in analyticity in a Bernstein ellipse $E_\rho$ with $\rho = 3 + 2\sqrt{2}$. One can therefore conclude that there exists a polynomial $b(x)$ with an approximation error of order $\mathcal{O}(\rho^{-N_2}) = \mathcal{O}(\exp(-N_2 \log{\rho}))$ with $\rho > 3 + 2\sqrt{2}$. 

In summary, there exists a lightning + polynomial approximation $r$ of the proposed form \cref{eq:ratapprox} with $a_j$, $p_j$ and $b(x)$ as defined above. If we take $$N_2 \geq \sqrt{N_th/4}\kern1pt/\log{\rho} = \mathcal{O}(\sqrt{N_1}),$$ $\lvert r_t(x) - r(x) \rvert = \mathcal{O}(\exp(-\sqrt{N_th/4}))$ is satisfied.
\end{proof}

Note that the change of variables $s + T = \sqrt{u}$ introduced in \cref{eq:inteq} depends on the truncation variable $T$. Therefore, the tapered distribution of the lightning poles \cref{eq:trapezoidal_poles} is intrinsically connected to numerical truncation. This connection and the notion of a truncation error also appears in the analysis of Stahl \cite{stahl1994rational} and Trefethen and collaborators \cite{trefethenExponentialNodeClustering2021}. 

One can now state that the lightning + polynomial approximation with $ \sigma = 2\sqrt{2}\pi$ converges to $\sqrt{x}$ at the rate of the best rational approximation.
\begin{theorem}\label{thm:convergence}
    There exist coefficients $\{a_j\}_{j=1}^{N_1}$ and a polynomial $b(x)$ of degree $N_2 = \mathcal{O}(\sqrt{N_1} \kern1pt)$, for which the degree $N$ lightning + polynomial approximation $r(x)$ \cref{eq:ratapprox} having tapered lightning poles \cref{eq:taperedpoles} with 
    \begin{equation}
        \sigma = 2\sqrt{2}\pi
        \label{eq:optsigmasqrtx}
    \end{equation}
    satisfies:
    $$\vert r(x) - \sqrt{x} \kern1pt \rvert = \mathcal{O}\left(e^{-\pi \sqrt{2N}} \kern1pt\right)$$
    as $N \to \infty$, uniformly for $x \in [0,1]$, under the same condition as \cref{thm:trap}.
\end{theorem}
\begin{proof} \sloppy
    From \cref{thm:trap,lemma:lp}, it follows that there exists a lightning + polynomial approximation $r$ of the proposed form satisfying:
    $$\vert r(x) - \sqrt{x} \kern1pt \rvert = \mathcal{O}\left(e^{-\pi\sqrt{2N_1}} \kern1pt\right).$$
    Furthermore, $r$ has total degree $N = N_1 + N_2 = \mathcal{O}(N_1)$, which leads to the given convergence behaviour.
\end{proof}

\subsection{Tapering at both ends}\label{sec:23}
Instead of a low-degree polynomial, i.e.~poles at infinity, one can include additional finite poles ``clustering towards infinity''. The best rational approximation to $|x|$ on $[-1,1]$ was studied by Stahl using tools of potential theory and asymptotic analysis~\cite{stahl1994rational} and all its poles are finite. As mentioned before, note that this problem is equivalent to the approximation of $\sqrt{x}$ on $[0,1]$. The asymptotic distribution of the poles away from $0$ is given explicitly in~\cite[Theorem 2.2]{stahl1994rational}. We can find explicit asymptotic estimates of the largest poles themselves from these results.

\begin{theorem}\label{thm:large_poles}
Let $p_{j,N}$, $j=1,\ldots,N$, be the poles of the degree $N$ best rational approximant to $\sqrt{x}$ on $[0,1]$, ordered by increasing absolute value. For fixed $k \geq 0$, the largest poles satisfy
 \begin{equation}\label{eq:large_poles}
  p_{N-k,N} \sim -\frac{8 N}{(2k+1)^2 \pi^2}, \qquad N \to \infty.
 \end{equation}
In particular, the largest pole satisfies $p_{N,N} \sim -\frac{8 N}{\pi^2}$.
\end{theorem}
\begin{proof}
    See \cref{app2}.
\end{proof}

Bounds are given for the small poles in~\cite{stahl1994rational}, but their distribution is not identified. Results for $x^\alpha$ in the later reference~\cite{stahlBestUniformRational2003} are not as explicit as those for $\sqrt{x}$ in~\cite[Theorem 2.2]{stahl1994rational}. Still, there is more to learn from~\cite{stahl1994rational}, such as an estimate of the number of poles larger (in absolute value) than $1$.

\begin{theorem}\label{thm:nb_of_poles}
The number of poles $p_{j,N}$ satisfying $|p_{j,N}| > 1$ scales like ${\mathcal O}(\sqrt{N} \kern1pt)$. In particular their total number is approximately $0.4 \sqrt{N}$.
\end{theorem}
\begin{proof}
    See \cref{app3}.
\end{proof}

%the contribution of the large poles is well approximated by a polynomial on $[0,1]$. Since such a polynomial approximation converges exponentially, as rational functions with poles in $(-\infty,-1]$ are analytic functions on $[0,1]$, a polynomial of degree $\mathcal{O}(\sqrt{N} \kern1pt )$ is sufficient to match the overall root-exponential convergence rate.

\Cref{fig:minimaxpoles} displays the poles of the best approximation of $\sqrt{x}$ on $[0,1]$. We calculated these by computing the best rational approximation of degree $2N$ of $\lvert x \rvert$ on $[-1,1]$ using Chebfun's \texttt{minimax} and squaring the poles obtained. One clearly identifies the tapered lightning poles. The large poles are also tapered, in the sense that the spacing between them increases even on this logarithmic scale as we approach the singularity at infinity. Note that although the minimax approximation has ${\mathcal O}(\sqrt{N} \kern1pt)$ finite poles with magnitude bigger than 1, we get the same asymptotic convergence rate when substituting them by poles at $-\infty$, as described in \cref{sec:22}.

\begin{figure}
    \centering
    \includegraphics[width=.6\linewidth]{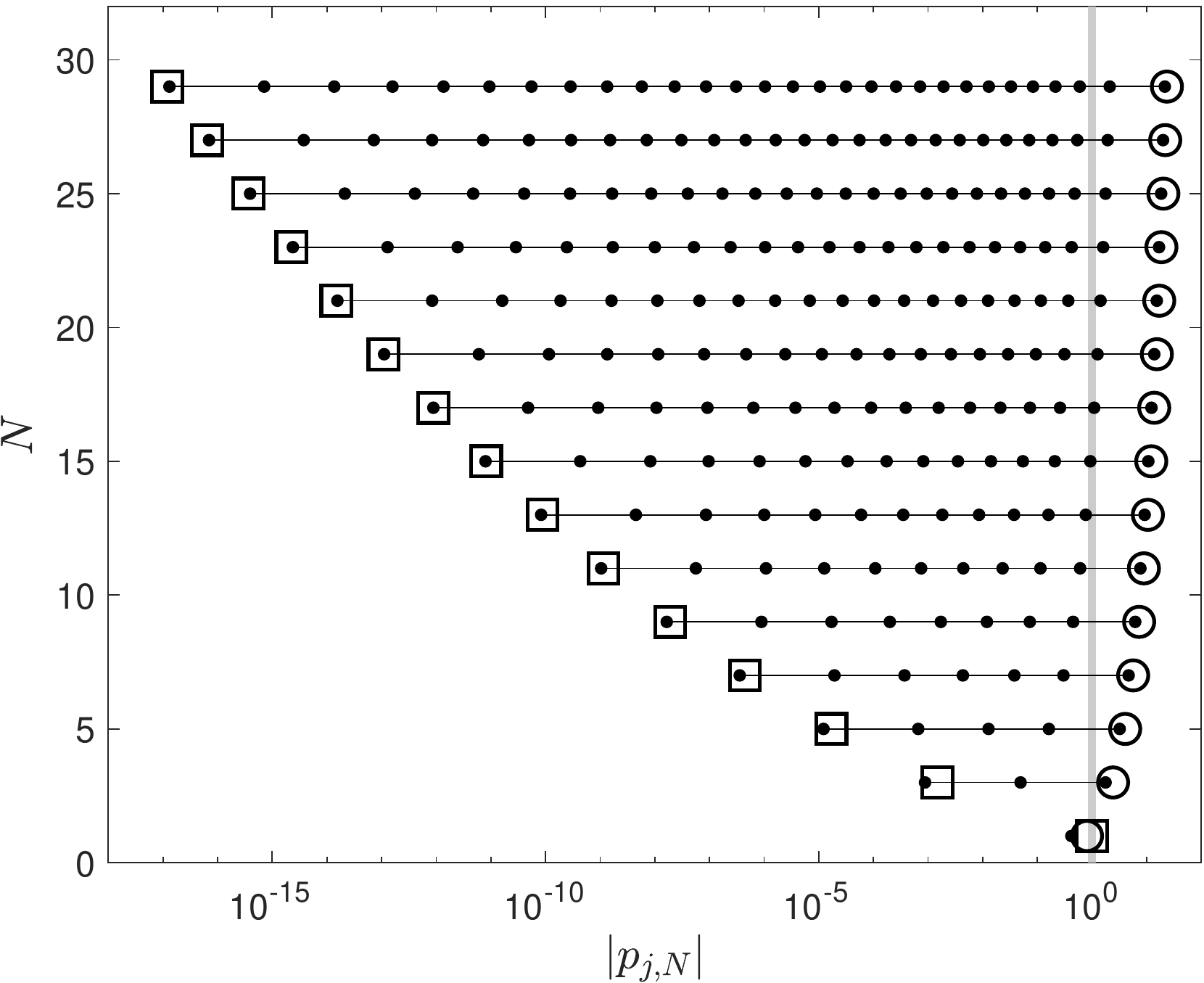}
    \caption{Magnitudes of the poles of the degree $N$ rational minimax approximation of $\sqrt{x}$ on $[0,1]$ for various $N$. Circles: $8N/\pi^2$ \cref{eq:large_poles}, squares: $\exp(-2\sqrt{2}\pi(\sqrt{N}-1))$ \cref{eq:taperedpoles}. The grey line at $\lvert p \rvert = 1$ emphasizes how few ``large poles'' there are with $\lvert p \rvert > 1$. }
    \label{fig:minimaxpoles}
\end{figure}

\section{\boldmath Approximation of $x^\alpha$ and $x^\alpha \log{x}$ on $[0,1]$}
\label{sec:3}
We do not know of an integral representation for $x^\alpha$ that leads to similar rational approximations. Yet the principles discussed in \cref{sec:2} apply to the approximation of branch point singularities in general. In this section we discuss the approximation of $x^\alpha$ (for $\alpha > 0$ and non-integer) and $x^\alpha \log(x)$ (for $\alpha > 0$) on $[0,1]$ numerically.

To this end, consider \cref{fig:twoparameter}, displaying the error of the least-squares approximation of $x^{\pi / 10}$ on $[0,1]$ using $N_1$ lightning poles and $N_2$ poles at infinity. A tapered exponential clustering (\ref{eq:taperedpoles}) on $[-1,0]$ with $\sigma = 2\sqrt{10\pi}$ is used. On the $x$-axis, the square root of the number of clustered poles $\sqrt{N_1}$ is displayed. The $y$-axis shows the degree of the added polynomial. One observes that the fastest convergence in $N$ is obtained for $N_2 \approx 1.1\sqrt{N_1} - 1$. This indicates that adding a polynomial of degree $N_2 = \mathcal{O}(\sqrt{N_1})$ is again optimal. Adding a polynomial with a higher degree does not improve the error.

\begin{figure}
    \centering
    \includegraphics[width=.6\linewidth]{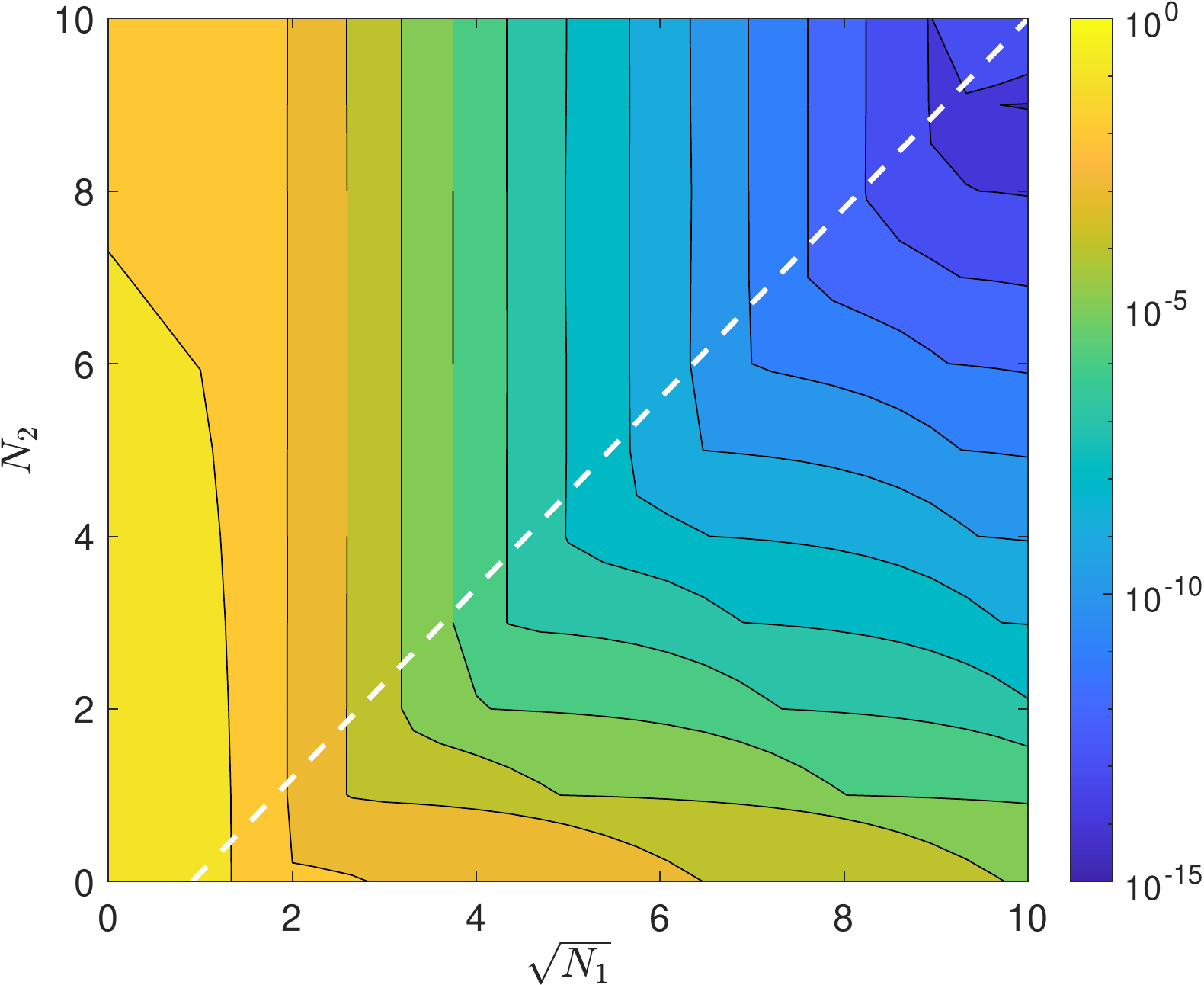}
    \caption{Max-norm error of the tapered lightning + polynomial approximation of $x^{\pi/10}$ on $[0,1]$ using $\sigma = 2\sqrt{10\pi}$. The dashed white line marks $N_2 = 1.1\sqrt{N_1} - 1$ and illustrates the regime of fastest convergence in $N$.}
    \label{fig:twoparameter}
\end{figure}

\Cref{fig:sigmadep} displays the errors of the lightning approximations of $x^{\pi / 10}$ on $[0,1]$ as functions of the clustering parameter $\sigma$. We use 10 partial fractions with tapered lightning poles \cref{eq:taperedpoles} on $[-1,0]$, a constant term and, optionally, a degree-3 polynomial. For the lightning + polynomial approximation the curve nicely illustrates the errors identified in the derivation of \cref{sec:22}.  With a fine spacing between the poles, the error is dominated by the truncation error, related to the smallest pole. In this regime, the errors of the lightning approximation behave almost identically. The pronounced V-curve of the lightning + polynomial approximation illustrates that a significant speedup can be obtained by optimizing the clustering parameter $\sigma$.

\begin{figure}
    \centering
    \includegraphics[width=.6\linewidth]{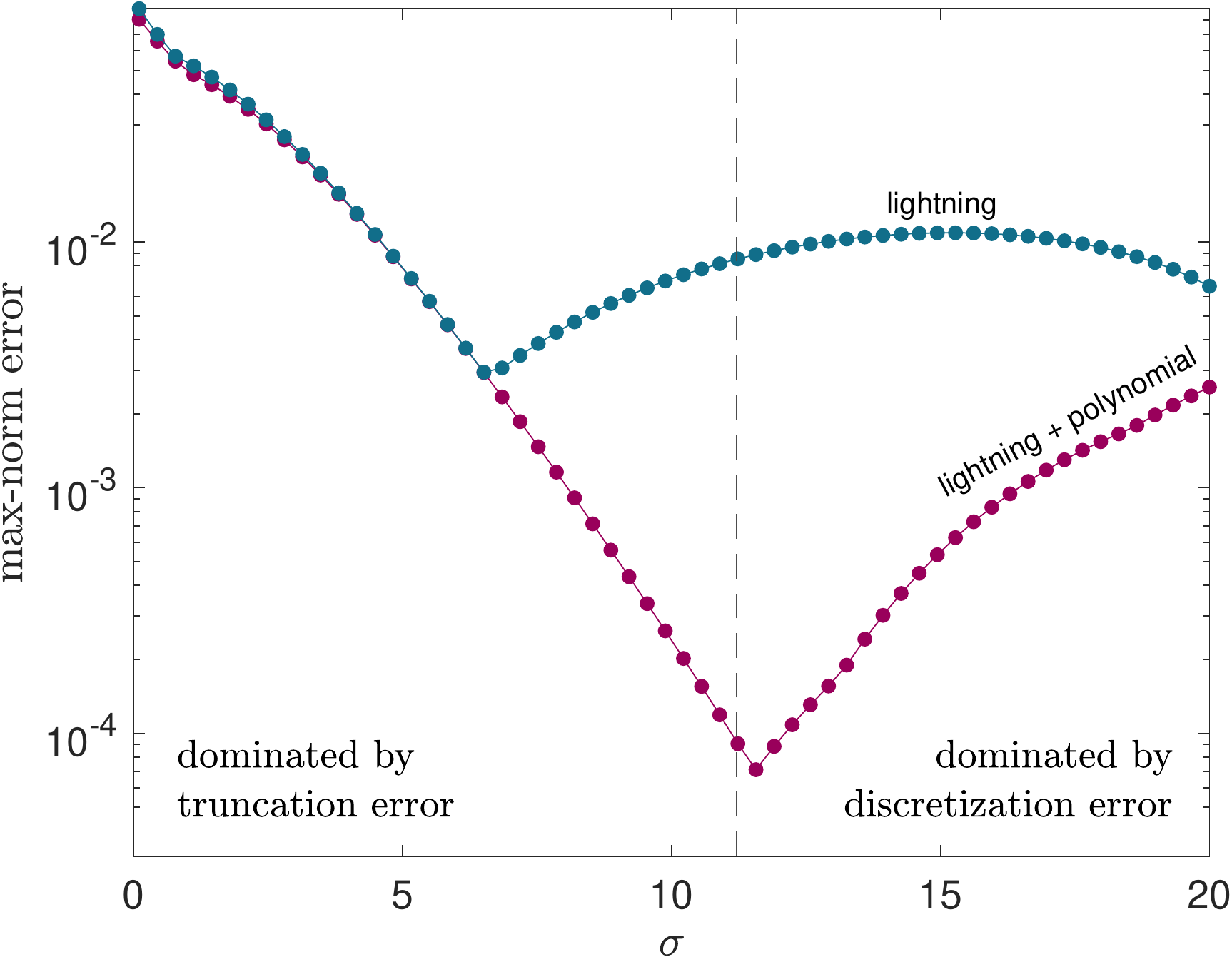}
    \caption{Max-norm error of the tapered lightning approximations as a function of the clustering parameter $\sigma$ for the approximation of $x^{\pi / 10}$ on $[0,1]$. The dashed line shows the conjectured result $\sigma = 2\pi/\sqrt{\alpha} = 2\sqrt{10\pi}$ \cref{conj:alpha}. The error regimes for the lightning + polynomial approximation follow from the derivation of \cref{sec:22}.}
    \label{fig:sigmadep}
\end{figure}

Since we do not know an integral representation for $x^\alpha$ from which to derive a similar rational approximation, the ``discretization error'' cannot be distinguished and an analogous error analysis as in \cref{sec:22} cannot be performed. Experimentally, one finds however that least-squares tapered lightning + polynomial approximations to $x^\alpha$ again converge at the rate of the minimax approximation for an optimal value of the spacing parameter $\sigma$. One can derive this optimal value for $\sigma$ by equating the truncation error, related to the size of the smallest pole, to the minimax error. The notion of a truncation error, as well as the idea of equating it to the best approximation error, follows from the work of Stahl \cite{stahlBestUniformRational1993, stahl1994rational, stahlBestUniformRational2003}:
\begin{equation}
    \text{Accuracy:} \qquad \varepsilon^\alpha \approx \exp(-\pi \sqrt{4 \alpha N} \kern1pt),
    \label{eq:optimalsigma}
\end{equation}
where $\varepsilon$ is the size of the smallest pole. For the tapered clustering (\ref{eq:taperedpoles}), we have $\varepsilon \approx \exp(-\sigma \sqrt{N})$. Therefore, assuming that the approximation achieves the best convergence rate, one can infer that the optimal clustering parameter $\sigma \approx 2\pi/\sqrt{\alpha}$. Numerical experiments show that the best approximation rate is indeed achieved for $\sigma = 2\pi/\sqrt{\alpha}$.
\begin{conjecture}
    There exist coefficients $\{a_j\}_{j=1}^{N_1}$ and a polynomial $b(x)$ with $N_2 = \mathcal{O}(\sqrt{N_1} \kern1pt)$, for which the lightning + polynomial approximation $r(x)$ \cref{eq:ratapprox} having tapered lightning poles \cref{eq:taperedpoles} with \begin{equation} \label{conj:alpha}
    \sigma = 2\pi/\sqrt{\alpha},
    \end{equation}
    satisfies:
    $$\vert r(x) - x^\alpha \kern1pt \rvert = \mathcal{O}\left(e^{-2\pi \sqrt{\alpha N}} \kern1pt\right)$$
    as $N \to \infty$, uniformly for $x \in [0,1]$.
\end{conjecture}

The minimizer of the V-curve displayed in \cref{fig:sigmadep} can be estimated numerically. \Cref{fig:typedependency} shows this empirically optimal value of $\sigma$ for the approximation of $x^\alpha$ on $[0,1]$, as a function of the type of the singularity $\alpha$. Here, $10$ tapered lightning poles and a degree-10 polynomial are used. The figure shows that the empirically optimal clustering parameter $\sigma$ is approximately equal to the conjectured value \cref{conj:alpha}. Analogous experiments show that the optimal clustering parameter $\sigma$ behaves similarly for logarithmic singularities of type $x^\alpha \log{x}$ ($\alpha > 0$). 

\begin{figure}
    \centering
    \includegraphics[width=.6\linewidth]{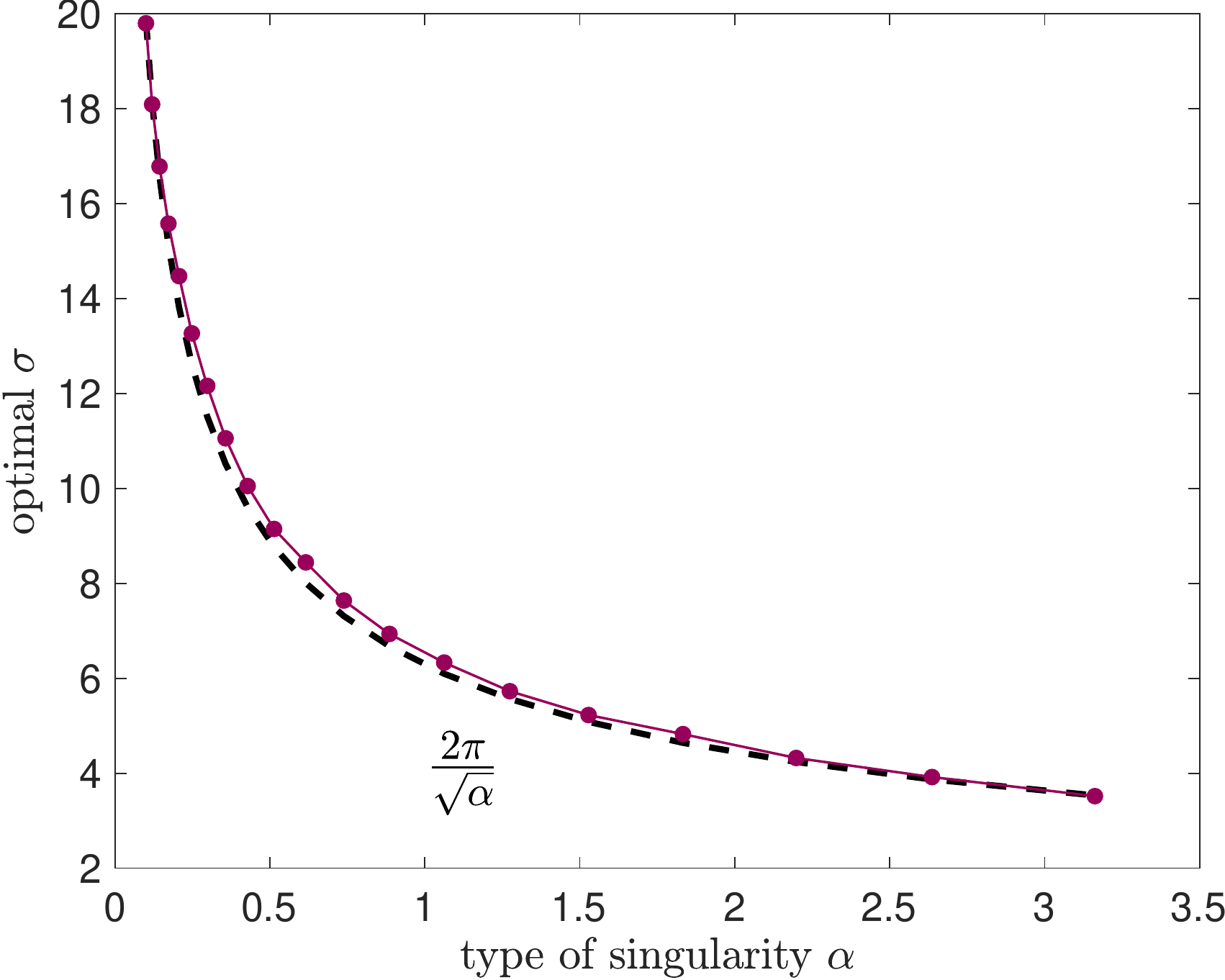}
    \caption{Optimal clustering parameter $\sigma$ as a function of the type of the singularity for the approximation of $x^\alpha$ on $[0,1]$. The dots display the value for each $\sigma$ that minimizes the max-norm error of the lightning + polynomial approximation using $N_1 = N_2 = 10$. The dashed line shows the conjectured value \cref{conj:alpha}.}
    \label{fig:typedependency}
\end{figure}

\section{Ill-conditioning and the size of the coefficients}
\label{sec:4}
The lightning method uses least-squares fitting to find accurate rational approximants:
\begin{equation}
    A\boldsymbol{c} \approx \boldsymbol{f}, \qquad A_{ij} = \phi_j(t_i), \; \boldsymbol{f}_i = f(t_i), 
    \label{eq:ls}
\end{equation}
with $\{t_i\}_{i=1}^M$ the sampling points in a domain $\Omega$ and the approximation set $\{\phi_j\}_{j=0}^N$ consisting of $N_1$ partial fractions with lightning poles $\{\frac{p_j}{z-p_j}\}_{j=1}^{N_1}$ and a polynomial basis of degree $N_2$. Note that the lightning basis functions are scaled to have unit max-norm on $[0,1]$. It is observed that the least-squares systems \cref{eq:ls} are in general heavily ill-conditioned. However, accurate approximations are still found using linear oversampling and standard regularization. 

This numerical phenomenon is related to redundancy in the approximation set, which was recently investigated in \cite{fna1,adcockFramesNumericalApproximation2020a}. The main conclusion of these works is that if the basis functions are sampled finely enough and the least-squares problem is solved with effective regularization at level $\epsilon$, then the computed fit has a residual that exceeds the mathematically minimal residual by an amount dominated by a term of order $\epsilon \norm{\mathbf{c}}_2$, where $\norm{\mathbf{c}}_2$ is the 2-norm of the vector of expansion coefficients. This theory is made precise for regularization by truncated singular value decomposition (TSVD) in \cite[Theorem 1.3]{adcockFramesNumericalApproximation2020a} by:
\begin{equation}
\small
    \norm{f-P^{\epsilon}_{M,N}f}_{\mathcal{H}} \leq \\ \inf_{\mathbf{c} \in \mathbb{C}^{N+1}} \left( \norm{f - \sum_{j=0}^N c_j\phi_j}_{\mathcal{H}} + \kappa_{M,N}^{\epsilon} \norm{A \mathbf{c} - \mathbf{f}}_2 + \epsilon \lambda_{M,N}^{\epsilon} \norm{\mathbf{c}}_2 \right)
    \label{eq:fna}
\end{equation}
in which $\kappa_{M,N}^{\epsilon}$ and $\lambda_{M,N}^{\epsilon}$ are coefficients that can be shown to be not very large if the samples are sufficiently dense. In the lightning Laplace code \cite{lightninglaplace}, linear oversampling by a factor of 3 is used and the sample points are exponentially clustered in a similar fashion as the lightning poles. For the numerical experiments on $[0,1]$ in the present paper, a fixed exponentially graded grid (\texttt{logspace(-16,0,2000)}) is used for simplicity.

The theory indicates that, in addition to inspecting the error, it is informative to inspect the norm of the coefficient vector. The trapezoidal rule construction in \cref{sec:22} shows that there exists an accurate lightning + polynomial approximation of $\sqrt{x}$ with bounded coefficients. This is the case generally for the approximation of branch point singularities, when the lightning poles are combined with a sufficient number of smooth basis functions. When the coefficients in the smooth basis are bounded as well, it now follows that accurate approximations can be found by least-squares fitting despite the ill-conditioning of the system matrix. In the lightning Laplace code \cite{lightninglaplace}, a discretely orthogonalized polynomial basis is used, constructed with the Vandermonde with Arnoldi algorithm \cite{vwa}. For the model problems in the present paper, associated with bounded intervals, we can use a basis of orthogonal polynomials scaled to that interval.

In \cref{fig:bigpoles}, the convergence behaviour as well as the 2-norm of the coefficient vector is compared for the tapered lightning approximations of $\sqrt{x}$ on $[0,1]$. The lightning poles are augmented with a polynomial term or poles clustering towards infinity, as discussed in \cref{sec:23}. The degree of the polynomial, as well as the number of large poles, is fixed at $N_2 = \texttt{ceil}(1.3 \sqrt{N_1}\kern1pt)$. The poles clustering towards infinity are constructed similarly to the asymptotics described in \cref{thm:large_poles}: 
$$p_i = -8N/((2i+1)^2\pi^2), \; 1 \leq i \leq N_2.$$ This approximation also converges at the best rate. For this experiment, a TSVD solver with threshold \texttt{2e-14} is used, resulting in an achievable error of order $\mathcal{O}(\epsilon \norm{\mathbf{c}}_2)$, as predicted by \cref{eq:fna}. Since the 2-norm of the coefficient vector is smallest for the lightning + polynomial approximation, it achieves the highest accuracy. Note that we do not know the optimal distribution of large poles for functions with more complex smooth behaviour, nor whether they result in bounded coefficients. We believe that an approximation set containing lightning poles and an orthogonal polynomial basis is a simple and robust choice.

\begin{figure}
\begin{minipage}{.47\linewidth}
    \centering
    \includegraphics[width=\linewidth]{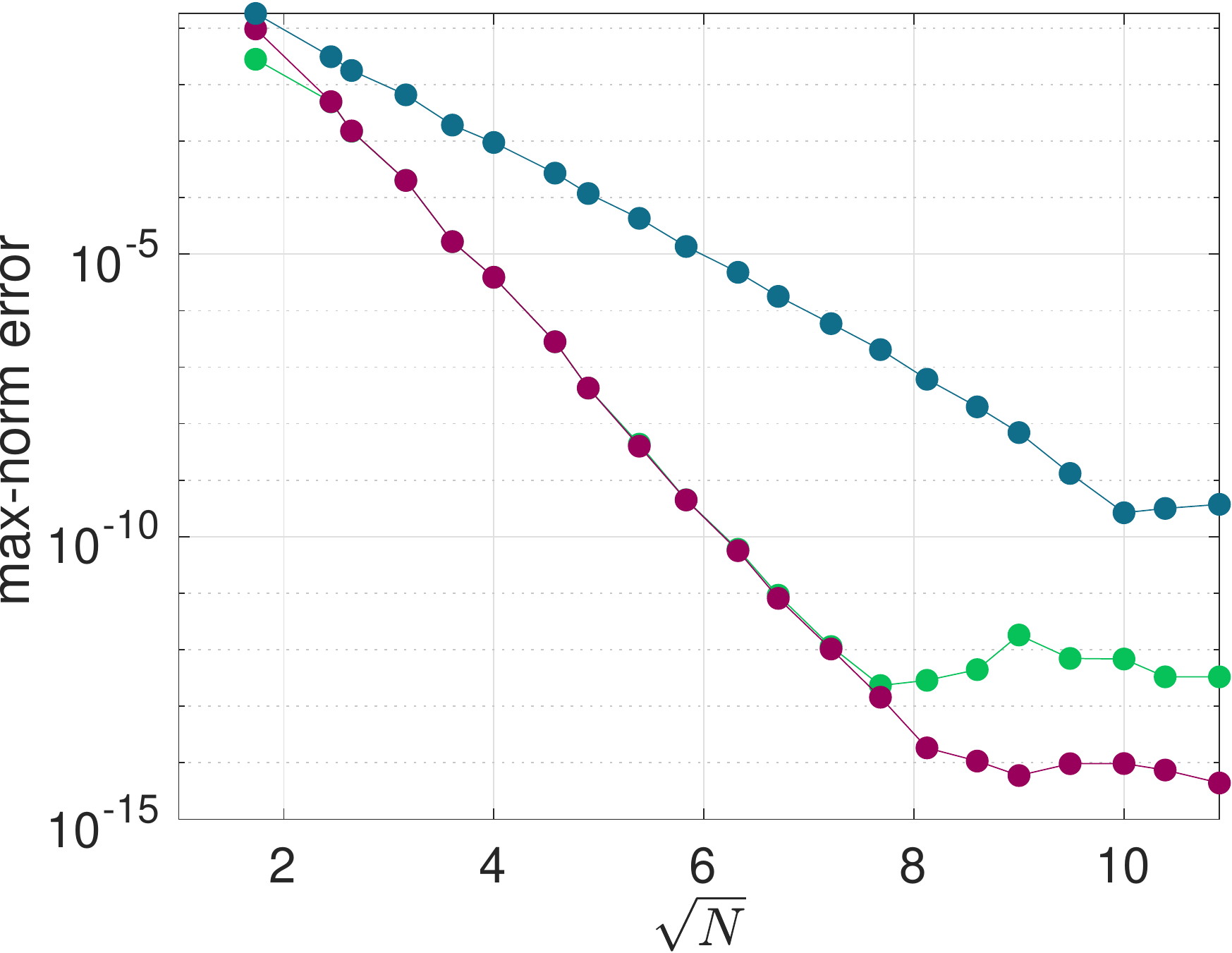}
\end{minipage}
\begin{minipage}{0.02\linewidth}
\hspace{0.02\linewidth}
\end{minipage}
\begin{minipage}{.47\linewidth}
    \centering
    \includegraphics[width=\linewidth]{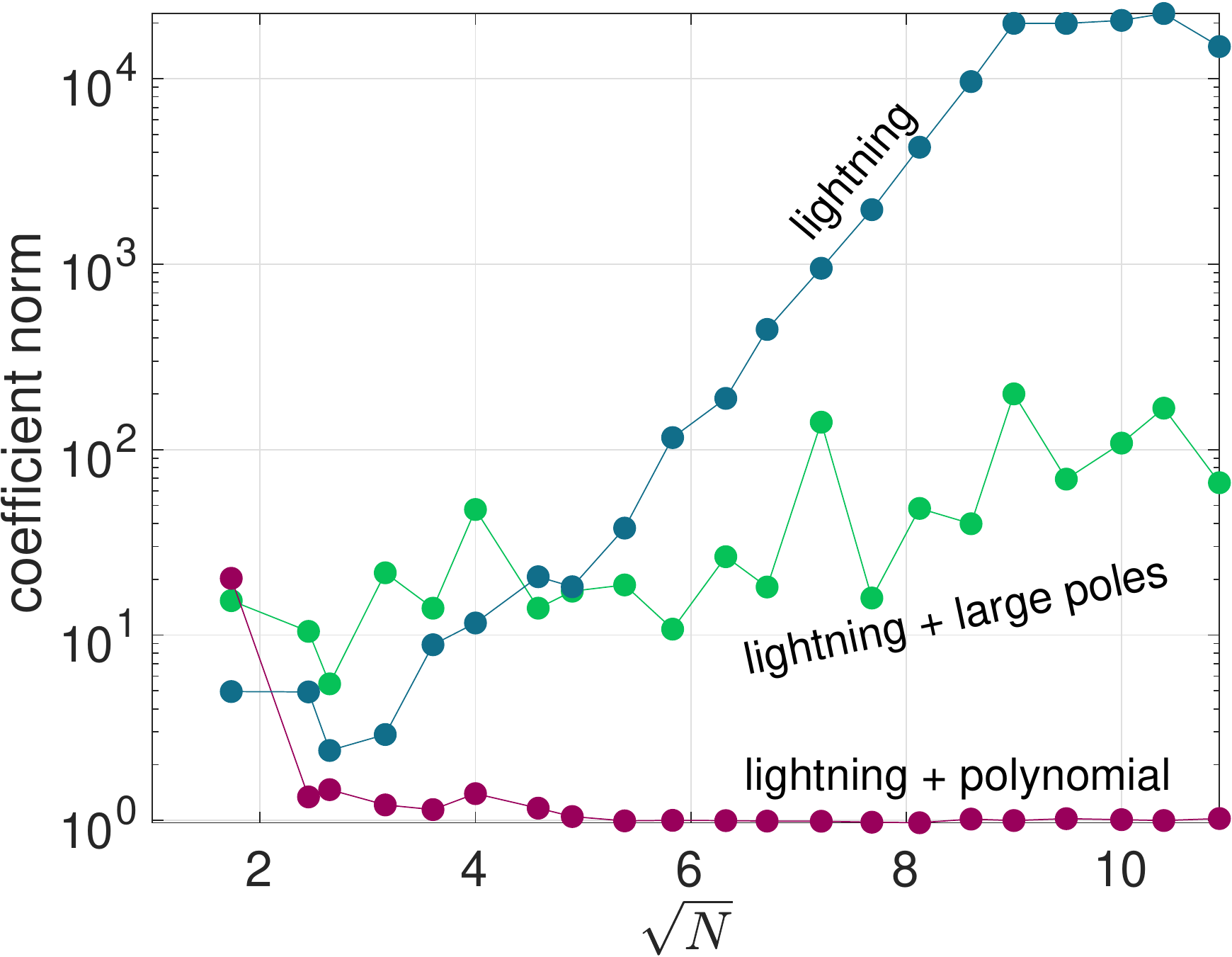}
\end{minipage}
\caption{Comparison of the lightning approximations of $\sqrt{x}$ on $[0,1]$. The lightning poles are augmented with a polynomial term or poles clustering towards infinity, as described in \cref{sec:22}. Left: max-norm error, right: $2$-norm of the coefficient vector.}
\label{fig:bigpoles}
\end{figure}

\section{More general approximation and PDE problems}
\label{sec:5}
In general, the lightning method can be used to approximate functions containing branch point singularities on curves in the complex plane. This idea led to the development of ``lightning solvers'', which approximate the solutions of PDEs on domains in the plane that are bounded by piecewise smooth Jordan curves with corners, such as polygons \cite{gopalSolvingLaplaceProblems2019}. These solutions exhibit singular behaviour near the corners. An example of such a domain is displayed in \cref{fig:house} (left). The lightning poles, marked by red dots, are exponentially clustered near the corners along the bisectors. Also, a global polynomial term in the complex variable has been included in the approximation set.

Trefethen and Gopal first developed a lightning method for the Laplace equation and provided an implementation in MATLAB \cite{lightninglaplace}. Based on numerical experiments, they proposed the parameter value $\sigma = 4$ in \cref{eq:taperedpoles} for the tapered distribution of the clustered poles near the corners. Later, this value for $\sigma$ was also used in \cite{trefethenExponentialNodeClustering2021}, an in-depth analysis of tapered pole distributions, and in \cite{brubecklightningStokesSolver2022}, a lightning method for 2D Stokes flow. This parameter value has proved strangely hard to beat. In this section, this phenomenon is analysed, resulting in an explanation for the near-optimality of $\sigma = 4$ for corner singularities of PDEs.

\begin{figure}
\begin{minipage}{.46\linewidth}
    \centering
    \includegraphics[width=\linewidth]{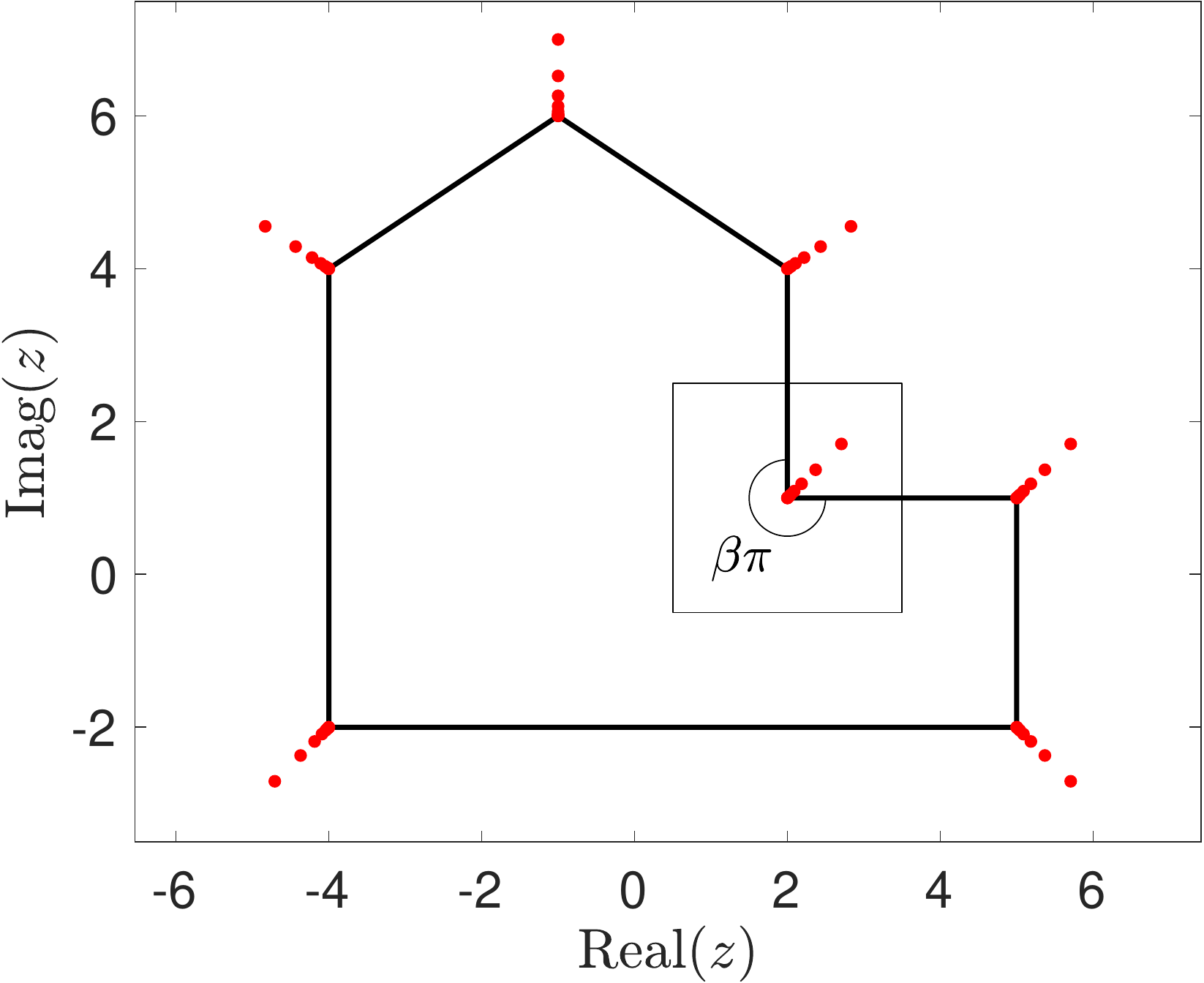}
\end{minipage}
\begin{minipage}{0.04\linewidth}
\hspace{0.04\linewidth}
\end{minipage}
\begin{minipage}{.46\linewidth}
    \centering
    \includegraphics[width=\linewidth]{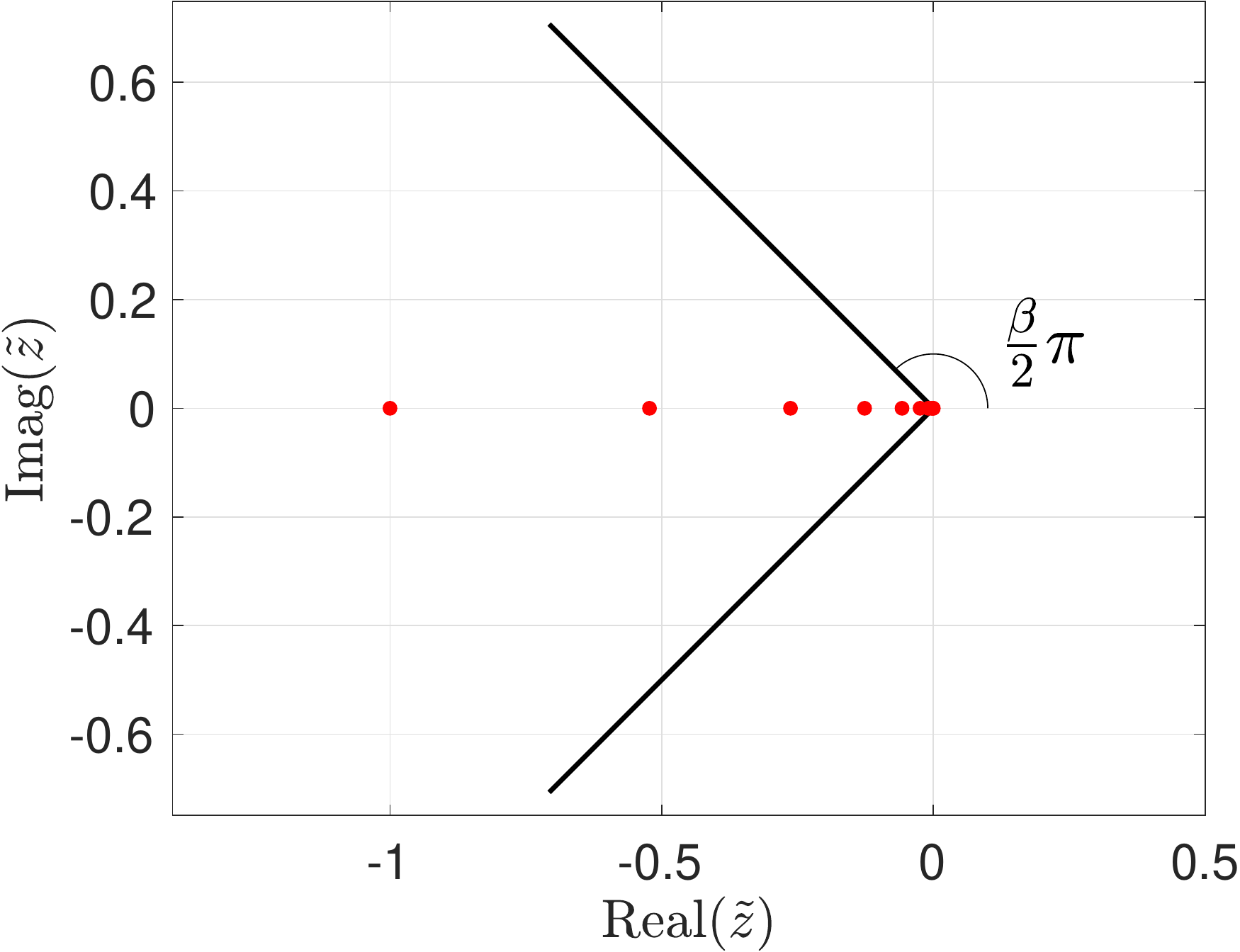}
\end{minipage}
\caption{A more general approximation problem in the complex plane (left) and the local approximation problem near the reentrant corner (right). The red dots mark the locations of the poles clustered near the corners of the computational domain. The local problem consists of an approximation domain composed of two unit intervals rotated by the angles $\beta\pi/2$ and $-\beta\pi/2$ in the plane, referred to as a V-shaped domain.}
\label{fig:house}
\end{figure}

\subsection{\boldmath Approximation of $x^\alpha$ on a V-shaped domain}

One can decompose the general approximation problem depicted in \cref{fig:house} (left) into multiple local approximation problems near the corners and a global smooth problem. For an explanation of how to make the decompositions rigorous by Cauchy integrals, see \cite[Thm.\ 2.3 and Fig.\ 3]{gopalSolvingLaplaceProblems2019}. \Cref{fig:house} (right) displays the local problem associated with the reentrant corner $\beta\pi$ of the original domain. Using a local complex variable, the poles clustering towards the corner are again positioned on the negative real axis. The problem greatly simplifies if the influence of the poles clustering towards other corners is thereby neglected and only the sample points close to the corner are taken into account. Note that the approximation interval has now been rotated to a pair of intervals closer to the poles. This has a significant influence on the achievable accuracy.

The problem in \cref{fig:house} (right) resembles the problems on the unit interval examined in previous sections. However, the approximation domain now consists of two unit intervals rotated by the angles $\beta\pi/2$ and $-\beta\pi/2$ in the plane. We will refer to this type of approximation domain as a ``V-shaped domain''. Again, we first analyse the lightning + polynomial approximation of $\sqrt{z}$ obtained using the construction introduced in \cref{sec:22}, based on applying the trapezoidal rule to \cref{eq:inteq}.

\begin{theorem} \sloppy
    The rational approximation \cref{eq:fulltrap} with $h = (2-\beta)\pi^2$ satisfies:
    $$\lvert r_t(z) - \sqrt{z} \kern1pt \rvert  = \mathcal{O}(e^{-\sqrt{N_th/4}}\kern1pt) = \mathcal{O}(e^{-\pi\sqrt{(2-\beta)N_t/4}} \kern1pt)$$
    as $N_t \to \infty$, uniformly for
    $z~\in~[0,1]e^{\pm i \beta\pi/2}, \; \beta \in [0,2)$, assuming the numerical bound introduced in \cref{conj:bound2} holds.
    \label{thm:trap2}
\end{theorem}
\begin{proof}
    See \cref{app4}.
\end{proof}

Similarly to \cref{thm:convergence}, this can be linked to the convergence behaviour of the lightning + polynomial approximation.

\begin{theorem}\label{thm:convergence2}
    There exist coefficients $\{a_j\}_{j=1}^{N_1}$ and a polynomial $b(z)$ of degree $N_2 = \mathcal{O}(\sqrt{N_1} \kern1pt)$, for which the lightning + polynomial approximation $r(z)$ \cref{eq:ratapprox} having tapered lightning poles \cref{eq:taperedpoles} with 
    \begin{equation} \label{eq:thm51sigma}
    \sigma = 2\sqrt{2-\beta}\pi
    \end{equation}
    satisfies:
    $$\vert r(z) - \sqrt{z} \kern1pt \rvert = \mathcal{O}(e^{-\pi \sqrt{(2-\beta)N}} \kern1pt)$$
    as $N \to \infty$, uniformly for
    $z \in [0,1]e^{\pm i \beta\pi/2}, \; \beta \in [0,2)$, under the same condition as \cref{thm:trap2}.
\end{theorem}

Empirically it is found that \cref{thm:convergence2} describes the optimal convergence rate for the lightning + polynomial approximation of $\sqrt{z}$. As an illustration, the errors for fixed degree lightning + polynomial approximations ($N_1 = 40$, $N_2 = 10$) of $\sqrt{z}$ are displayed in Figure \ref{fig:angledep} as a function of the angle $\beta\pi$ of the V-shaped domain. The tapered lightning poles are positioned on $[-1,0]$ and results are computed for $\sigma = 2\sqrt{2}\pi$ \cref{eq:optsigmasqrtx} (circles), $\sigma = 2\sqrt{2-\beta}\pi$ \cref{eq:thm51sigma} (dots) and $\sigma = 4$ (crosses). The approximation using $\sigma = 2\sqrt{2}\pi$ is only optimal for $\beta = 0$, i.e.\ approximation on the unit interval. However, for a V-shaped domain, the error using $\sigma = 2\sqrt{2-\beta}\pi$ \cref{eq:thm51sigma} is found to be optimal. The figure reveals that the achievable accuracy rapidly decreases as $\beta \pi$ increases.

These results for $\sqrt{z}$ lead one to suspect that the value given in \cref{eq:cnj52sigma} below is an optimal choice of $\sigma$ for the approximation of $z^\alpha$ on a V-shaped domain. Numerical experiments indeed indicate that the results can again be generalized from $\sqrt{z}$ to $z^\alpha$.
\begin{conjecture}\label{conj:rotated}
    There exist coefficients $\{a_j\}_{j=1}^{N_1}$ and a polynomial $b(z)$ of degree $N_2 = \mathcal{O}(\sqrt{N_1} \kern1pt)$, for which the lightning + polynomial approximation $r(z)$ \cref{eq:ratapprox} having tapered lightning poles \cref{eq:taperedpoles} with 
    \begin{equation} \label{eq:cnj52sigma}
    \sigma = \sqrt{2(2-\beta)}\pi/\sqrt{\alpha}
    \end{equation}
    satisfies
    $$\vert r(z) - z^\alpha \kern1pt \rvert = \mathcal{O}(e^{-\pi \sqrt{2(2-\beta)\alpha N}} \kern1pt)$$
    as $N \to \infty$, uniformly for
    $z \in [0,1]e^{\pm i \beta\pi/2}, \; \beta \in [0,2).$
\end{conjecture}
The optimal value for $\sigma$ is again found to behave similarly for the approximation of $z^\alpha \log{z}$.

\begin{figure}
    \centering
    \includegraphics[width=0.7\linewidth]{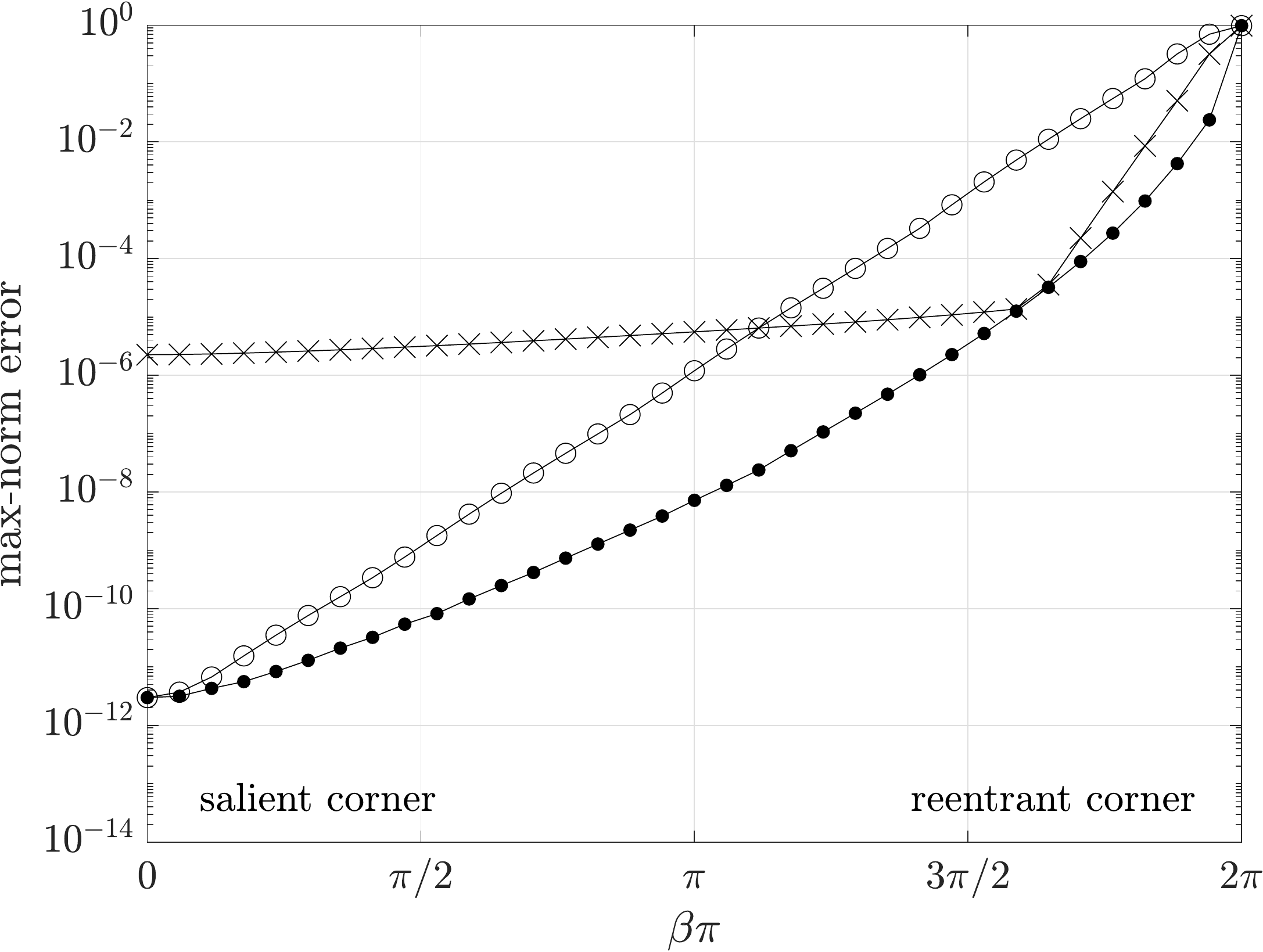}
    \caption{Max-norm errors of the tapered lightning + polynomial approximations ($N_1 = 40$, $N_2 = 10$) of $\sqrt{z}$ as functions of the corner of the V-shaped domain. Circles: $\sigma = 2\sqrt{2}\pi$ \cref{eq:optsigmasqrtx}, dots: $\sigma = 2\sqrt{2-\beta}\pi$ \cref{eq:thm51sigma}, crosses: $\sigma = 4$.}
    \label{fig:angledep}
\end{figure}
 
\subsection{Lightning PDE solvers}
For the Laplace PDE, the link between the type of the corner singularity $z^\alpha$ and the angle of the corner $\beta \pi$ is described in \cite{wasow}. Following \cite[Theorem 5]{wasow}, the dominant asymptotic behaviour near the corner can be described as $\mathcal{O}(z^{1/\beta})$ for $1/\beta$ non-integer and $\mathcal{O}(z^{1/\beta}\log z)$ for $1/\beta$ integer, $\beta \in (0,2)$. The local Laplace problem near a corner $\beta \pi$ is therefore equivalent to an approximation problem of $z^\alpha$ or $z^\alpha\log{z}$ with $\alpha = 1/\beta$ on a V-shaped domain. The conjectured value \cref{eq:cnj52sigma} implies an optimal value for $\sigma$:
\begin{equation}
    \sigma = \sqrt{2(2 - \beta)\beta}\pi,
    \label{eq:optsigmaalles}
\end{equation}
with $\beta \in (0,2)$. 

\Cref{fig:typedep2} compares this result to the empirically optimal value for $\sigma$. It displays the value of $\sigma$ that minimizes the error of the fixed degree lightning + polynomial approximation ($N_1 = N_2 = 20$) of $z^{1/\beta}$ on a V-shaped domain with corner $\beta \pi$. For $\beta \pi \to 0$, a salient corner, the singularity $z^\alpha$ becomes increasingly weak, since $\alpha \to \infty$. Also, the V-shaped approximation domain lies close to the positive real axis, away from the lightning poles. In this regime, the approximation quickly converges to a level close to machine precision. For $\beta \pi \to 2\pi$, a reentrant corner, the domain converges to a domain containing a slit. The singularity $z^\alpha$ becomes increasingly strong, since $\alpha \to 0$, while the V-shaped domain lies close to the poles on the negative real axis. In this regime, the approximation converges very slowly. For the wide range of corners $\beta \pi$ between these extremes, it follows that $\sigma = 4$ is a good approximation to the optimal $\sigma$. Therefore, the value $\sigma = 4$ used in \cite{brubecklightningStokesSolver2022,gopalSolvingLaplaceProblems2019,trefethenExponentialNodeClustering2021} results in root-exponential convergence at a near-optimal rate for corner singularities.

\begin{figure}
    \centering
    \includegraphics[width=.65\linewidth]{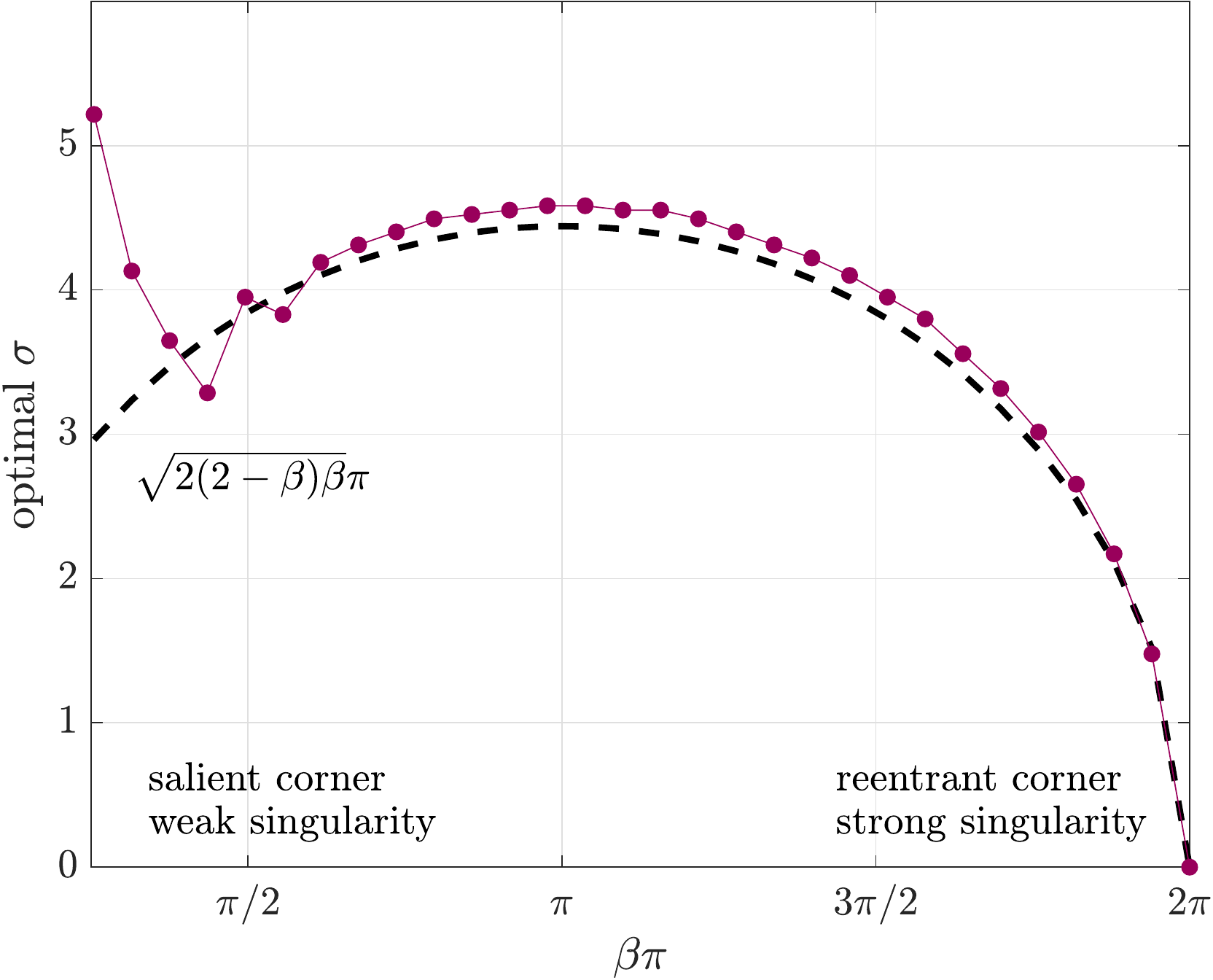}
    \caption{Optimal clustering parameter $\sigma$ for the approximation of $x^{1/\beta}$ on a V-shaped domain with corner $\beta \pi$. This problem models the local Laplace problem near a corner. The dots display the values for $\sigma$ that minimize the error of the tapered lightning + polynomial approximation using $N_1 = N_2 = 20$. The dashed line marks the conjectured value \cref{eq:optsigmaalles}.}
    \label{fig:typedep2}
\end{figure}

\appendix
\section{Proof of \cref{thm:trap}}\label{app1}
We refer to the following formulas:
\begin{equation}\label{app:f}
f(u,x) = \frac{x}{\pi}\frac{1}{\sqrt{u}} \frac{e^{\sqrt{u}-T}}{e^{2(\sqrt{u}-T)}+x}.
\end{equation}

\begin{equation}\label{app:I}
     I(x) = \frac{2x}{\pi} \int_{-T}^{T} \frac{e^s}{e^{2s}+x} \kern2pt ds = \int_{0}^{4T^2} f(u,x) \kern2pt du
\end{equation}

\begin{equation}\label{app:S}
     S(x) = h \sum_{j=1}^{N_t} f(jh, x).
\end{equation}
Thus, we choose $T = \sqrt{N_t h/4}$.

\begin{lemma}
 The truncation error for $T > 0$ and for any $x \in [0,1]$ satisfies
 \[
  |\sqrt{x} - I(x)| < \frac{4}{\pi} e^{-T} < \frac32 e^{-T}.
 \]
\end{lemma}
\begin{proof}
 We have
 \[
 \sqrt{x} - I(x) = \frac{2x}{\pi} \int_{-\infty}^{\infty} \frac{e^s}{e^{2s}+x} \kern2pt ds - \frac{2x}{\pi} \int_{-T}^{T} \frac{e^s}{e^{2s}+x} \kern2pt ds.
 \]
We can bound the integrand by omitting $x$ from the denominator for $s \geq T$, and omitting $e^{2s}$ from the denominator for $s \leq -T$. Exponential decay in either direction away from $\pm T$ shows that both tail integrals are bounded by the value of the integrand at the endpoints $\pm T$. Noting in addition that $x\leq 1$, the result follows.
\end{proof}

For small $x$, we evaluate the quadrature error directly.
\begin{lemma}
 For $x \in [0,e^{4-2T}]$ the quadrature error satisfies
 \[
  |I(x) - S(x)| < 2 e^{2-T} < 15 e^{-T}.
 \]
\end{lemma}
\begin{proof}
Because the integrand of its integral representation is positive, by truncation we have $I(x) < \sqrt{x}$. Thus, for $x \in [0,e^{4-2T}]$ we have $I(x) < e^{2-T}$. For $S(x)$, straightforward but lengthy calculations show that the integrand $f(u,x)$ is monotonically decreasing, for fixed $x$ and as a function of $u \in [0,\infty)$, as long as $x \leq x^* = e^{2\sqrt{2}-2T} \frac{\sqrt{2}+1}{\sqrt{2}-1} \approx e^{4.59-2T}$. Thus, in the range of $x$ of this lemma, $S(x)$ is a Riemann sum approximation of $I(x)$ based on samples at the right of each rectangle. Because the integrand is positive and decays monotonically this implies $S(x) < I(x)$.
\end{proof}

Next, we formulate an exact representation of the quadrature error in terms of a contour integral. To that end, following \cite{trefethen2014trapezoidal}, we define the function
\begin{equation}
 \delta(u) = \mu(u) - m(u), \label{delta}
\end{equation}
with
\[
\mu(u) = 
\begin{cases}
-\frac12 ~~~\Im u \ge 0, \\[2pt]
\hphantom{-}\frac12 ~~~\Im u < 0,
\end{cases}
\]
and
\[
m(u) = -\frac{i}{2} \cot\left( \frac{\pi u}{h} \right).
\]
These formulas correspond to those given in Table 13.1 and Table 13.2 of~\cite{trefethen2014trapezoidal}.

First, we consider the possible poles of the integrand.
\begin{lemma}\label{lem:poles}
 For $x \in (0,1]$, as a function of $u$ the function $f(u,x)$ has poles at
 \[
  \pi_{\pm,k} = \left(T + \frac12 \log x\right)^2 - \pi^2 \left(k + \frac12\right)^2 \pm i \pi \left(T \left(2k+1\right) + \left(k+\frac12\right) \log x\right).
 \]
 The residues of the poles $\pi_{\pm} = \pi_{\pm,0}$ closest to the real axis, for $k=0$, are $r_{\pm} = \frac{\mp i\sqrt{x}}{\pi}$.
\end{lemma}
\begin{proof}
The poles correspond to the roots of the denominator of $f(u,x)$. That leads to $2(\sqrt{u}-T) = \log(-x)$. The result follows by taking into account all branches of the logarithm $\log(-x) = \log x \pm (2k + 1) \pi i, k=0,1,2,\dots$. The residues for $k=0$ are readily obtained by further calculation.
\end{proof}

\begin{theorem}\label{thm:contour}
 For $x \in [e^{4-2T},1]$, the quadrature error is given by
 \[
 I(x) - S(x) = \int_{\Gamma_1 \cup \Gamma_2} f(u,x) \kern2pt du + \int_\Gamma f(u,x) \delta(u) \kern2pt du - 2\pi i (r_{+} \delta(\pi_{+}) + r_- \delta(\pi_-)),
 \]
 in which $\Gamma_1 = [0,1]$, $\Gamma_2 = [4T^2+1,4T^2]$ and $\Gamma$ corresponds to the positively oriented rectangle $[1,4T^2+1] \times [-ai,ai]$, with $a = 2\pi (T + \frac12 \log x)$. Finally, $r_{\pm}$ and $\pi_{\pm}$ are the poles and residues of $f(u,x)$ according to Lemma~\ref{lem:poles}.
\end{theorem}
\begin{proof}
Let us first show that the rectangle delineated by $\Gamma$ contains just two poles, and they are precisely $\pi_{\pm}$. Indeed, from Lemma~\ref{lem:poles}, for $x \in [e^{4-2T},1]$ the real part of the poles $\pi_{\pm}$ are larger than $4-\pi^2/4 \approx 1.53$ and smaller than $T^2-\pi^2/4$. Their imaginary part is $\pm a/2$. All other poles have imaginary part greater than $a$ (in absolute value) for any $k \neq 0$, unless possibly for $x < e^{4-2T}$. Both cases lie outside the rectangle.

We can write $I(x)$ as the sum of the integral of $f$ on $\Gamma_1 \cup \Gamma_2$ and
 \[
 I_{\textrm{mid}}(x) = \int_1^{4T^2+1} f(u,x) du  = \int_\Gamma f(u,x) \mu(u)du - 2 \pi i (r_+ \mu(\pi_+) + r_- \mu(\pi_-)).
\]
The latter equality follows from the construction of $\mu(u)$ as the characteristic function of a bounded interval, see~\cite[\S 13]{trefethen2014trapezoidal}, and subtracting out the residues that were picked up at the two poles after deformation onto $\Gamma$.

A similar expression holds for $S(x)$, involving $m(u)$ rather than $\mu(u)$. However, the function $m(u)$ also has poles on the real line at $j h = j 4T^2/N_t$, for $j \in \mathbb{Z}$, and that set includes the endpoints $0$ and $4T^2$. The range of integration was shifted from $[0,4T^2]$ to $[1,4T^2+1]$ in order to avoid those. Note that the remaining poles of $m$ contained within $\Gamma$ correspond precisely to the quadrature points, and their residues sum up to the trapezoidal rule applied to $f$, see~\cite[\S 13]{trefethen2014trapezoidal}. The final expression follows from $\delta(u) = \mu(u) - m(u)$.
\end{proof}

Since the integrand function $f(u,x)$ is small near the endpoints of $[0,4T^2]$, for small values of $N_t$ the size of the residues is the dominant factor in the discretization error. That leads to the following optimal choice of $h$, with which the convergence rate of the best rational approximation to $\sqrt{x}$ is obtained.

To that end, we first wish to quantify the size of $\delta(u)$ in the complex plane.

\begin{lemma}\label{lem:delta}
For any $u \in \mathbb{C}$ with $| \Imm u | \geq h/(2\pi)$,
\[
|\delta(u)| \leq \frac32 e^{-2\pi |\Imm u|/h}.
\]
% Along the lines $1+kh \pm si$, for each $k \in \mathbb{Z}$, {\color{magenta} gebruiken we dit?}
% \[
% |\delta(u)| \leq 4 e^{-2\pi |s|/h}.
% \]
\end{lemma}
\begin{proof}
We need to show that for $| \Imm u | \geq h/(2\pi)$,
\[
\left| -\frac{1}{2} + \frac{i}{2}\cot{\left(\frac{\pi u}{h}\right)} \right| \leq \frac{3}{2} e^{-2\pi \left| \Imm u \right|/h}.
\]
Equivalently we must show that for $\Ree{u} \geq \frac{1}{2}$,
\[
\left| \kern1pt \coth{(u)} - 1 \kern1pt \right| \leq 3e^{-2\Ree{u}}.
\]
This follows from the calculation
\[
\coth{(u)} = \frac{e^u + e^{-u}}{e^u - e^{-u}} = 1 + \frac{2e^{-2u}}{1-e^{-2u}}.
\]
\end{proof}

\begin{theorem}
 The truncation error $\sqrt{x}-I(x)$ and the size of the residues $r_{\pm} \delta(\pi_{\pm})$ in Theorem~\ref{thm:contour} decay at the same rate in $N_t$, independently of $x \in [0,1]$, with the choice
 \[
  h = 2\pi^2
 \]
and that rate is $e^{-T} = e^{-\pi \sqrt{N_t/2}}$.
\end{theorem}
\begin{proof}
It suffices to calculate the residue of $\delta(u)$ at the pole $\pi_{+}$. We know from Lemma~\ref{lem:poles} that $r_+ = -i \sqrt{x}/\pi$. By Lemma~\ref{lem:delta}, $|\delta(\pi_+)| \leq \frac{3}{2} e^{-2 \pi \Imm \pi_+ / h}$. Thus, since $e^{\frac12 \log x} = \sqrt{x}$,
\[
 | r_+ \delta(\pi_+) | \leq \frac{3}{2\pi}\sqrt{x} e^{-2 \pi^2 (T + \frac12 \log x)/ h} = \frac{3}{2\pi}\sqrt{x} e^{-2 \pi^2 T/h} (\sqrt{x})^{-2 \pi^2/ h}.
\]
This matches the rate $e^{-T}$ of the truncation error when $h = 2\pi^2$, independently of $x$.
\end{proof}

When $h = 2\pi^2$, the last term in Theorem~\ref{thm:contour} can be bounded by
\begin{equation*}
    \lvert 2 \pi i (r_+ \mu(\pi_+) + r_- \mu(\pi_-)) \rvert \leq 6 \kern1pt e^{-T}.
\end{equation*}
To conclude the proof, the integrals introduced in \cref{thm:contour} remain to be bounded. It is easy to show that 
\[ 
\left| \int_{\Gamma_1 \cup \Gamma_2} f(u,x) \kern2pt du \right| < \frac{3}{2} \kern1pt e^{-T}.
\]
Furthermore, strong numerical evidence indicates that the remaining integral can be bounded in the following way.
\begin{conjecture}
    For $x \in [e^{4-2T},1]$
    \[
    \left| \int_\Gamma f(u,x) \delta(u) \kern2pt du \right| < 12 \kern1pt e^{-T}
    \]
    in which $\Gamma$ corresponds to the positively oriented rectangle $[1,4T^2+1] \times [-ai,ai]$, with $a = 2\pi (T + \frac12 \log x)$.
    \label{conj:bound}
\end{conjecture}
We numerically verified the size of the integral for $x$ in $[0,1]$ for both small and large values of $N_t$. A complete analytic understanding of the different regimes involved would lead us too far.

\section{Proof of \cref{thm:large_poles}}\label{app2}
\begin{proof}
 The poles $\pi_{j,N}$ of the best approximant to $|x|$ on $[-1,1]$ are described in~\cite[Theorem 2.2]{stahl1994rational}, asymptotically for large $N$, in terms of the function
\begin{equation}\label{eq:H}
    H_N(y) = \frac{N+1}{2} - \frac{1}{\pi} \int_y^\infty \left[ \frac{\sqrt{N}}{t \sqrt{1+t^2}}+\frac{1}{\pi t} \log \left( \frac{t}{1 + \sqrt{1+t^2}} \right) \right] {\rm d}t.
\end{equation}
This function is invertible for $y > 0$ and $N$ sufficiently large, and the poles satisfy
\[
\pi_{j,N} \sim i H_N^{-1}(j).
\]
They are all on the imaginary axis, symmetric with respect to the real line. The poles $p_{j,N}$ of the best approximant to $\sqrt{x}$ are the squares of those of $|x|$~\cite{newman1964}, and hence lie on the negative real axis. In the notation of~\cite{stahl1994rational}, the precise correspondence is $p_{j,N} = \pi_{j,2N}^2$, $1\leq j \leq N$. In our proof it is sufficient to estimate the inverse of $H_N(y)$ for large $N$, evaluated at an integer.

We first estimate $H_N(y)$ itself, assuming large argument $y$. We write
\[
 H_N(y) = \frac{N+1}{2} - \sqrt{N} F_1(y) - F_2(y),
\]
with
\begin{equation}\label{eq:F1}
 F_1(y) = \frac{1}{\pi} \int_y^\infty \frac{1}{t \sqrt{1+t^2}} \kern1pt {\rm d}t
\end{equation}
and
\begin{equation}\label{eq:F2}
F_2(y) = \frac{1}{\pi^2} \int_y^\infty \frac{1}{t} \log \left( \frac{t}{1 + \sqrt{1+t^2}} \right) {\rm d}t.
\end{equation}
Note that $F_2(y)$ is independent of $N$ and, hence, $H_N(y) \sim \frac{N+1}{2} - \sqrt{N} F_1(y) + {\mathcal O }(1)$.

We expand $F_1(y)$ for large $y$. Expanding the integrand as
\[
 \frac{1}{t \sqrt{1+t^2}} = \frac{1}{t^2} + \frac{1}{2t^2} + {\mathcal O}(t^{-6})
\]
and integrating term by term yields
\[
 F_1(y) = \frac{1}{\pi y} - \frac{1}{4 \pi y^3} + {\mathcal O}(y^{-5}).
\]
Thus, to leading order, we find that
\[
 H_N(y) \sim \frac{n+1}{2} - \frac{\sqrt{N}}{\pi y}
\]
with approximate inverse
\[
 H_N^{-1}(x) \sim \frac{\sqrt{N}}{\pi} \left(\frac{N+1}{2}-x \right)^{-1}.
\]
Evaluating at $j = N-k$ leads to the expression
\[
\pi_{N-k,N} \sim i\frac{2\sqrt{N}}{(2k+1)\pi}.
\]
Since the poles grow with $N$, the assumption above of large $y$ is justified. Finally, the poles for $\sqrt{x}$ are obtained by squaring and substituting $2N$ for $N$. This leads to the result.
\end{proof}

\section{Proof of \cref{thm:nb_of_poles}} \label{app3}
\begin{proof}
We start by observing that~\eqref{eq:large_poles} is only accurate for larger poles (so sufficiently small $k$), since in the analysis of the explicit density~\eqref{eq:H} we have assumed large $y$, i.e., a large pole. However, \eqref{eq:H} itself is valid for small $y$ too, as long as $H_N(y)$ is invertible. This is the case for every $y>0$ and $N$ sufficiently large.

Since $p_{j,N} = \pi_{j,2N}^2$, we are looking for the index $j$ such that $\pi_{j,N} \approx i$. In particular, we want to find the largest $k$ such that $i^{-1}\pi_{N-k,2N} > 1$, corresponding to $p_{N-k,N} < -1$. From $\pi_{j,N} \sim iH_N^{-1}(j)$ it follows that $H_{N}(i^{-1}\pi_{j,k}) = j$. It suffices to estimate $H_N(1)$.

The expression~\eqref{eq:H} is fully explicit and we find
\[
 H_{N}(1) = \frac{N+1}{2} - \sqrt{N} F_1(1) - F_2(1).
\]
Identifying $N-k$ with $H_{2N}(1) \sim N - \sqrt{2N} F_1(1)$, it follows at once that $k = {\mathcal O}(\sqrt{N} \kern1pt)$. The numerical value is obtained by evaluating the integral expression for $F_1(1) = \frac{1}{\pi} \int_1^\infty \frac{1}{t \sqrt{1+t^2}} {\rm d}t \approx 0.28$. This leads to $H_{2N}(1) \approx N - 0.4\sqrt{N}$.
\end{proof}

\section{Proof of \cref{thm:trap2}}\label{app4}
We concisely state a generalization of the proof given in \cref{app1} for the V-curved domain: $z = r \exp{(\pm \beta \pi i /2 )}$ with fixed $\beta \in [0,2)$ and $ r\in [\kern .3pt 0,1]$. The formulas \cref{app:f}, \cref{app:I}, \cref{app:S} and \cref{delta} are defined identically, substituting $x$ with the complex variable $z$. The proofs of these statements are analogous to those of the simpler case $\beta=0$, albeit technically more involved. The parameter $\beta$ influences the location of the poles, and correspondingly leads to a slightly different integration contour. In all these expressions the limit $\beta \to 0$ agrees with the earlier results, but the limit $\beta \to 2$ is not viable. The latter limit corresponds to the degenerate case of an angle of $2\pi$.

\begin{lemma} \sloppy
 The truncation error for $T > \frac{1}{2}\log{2} \approx 0.35$ and for $z = r \exp{(\pm \beta \pi i /2 )}$ with fixed $\beta \in [0,2)$ and $r \in [0,1]$ satisfies
 \[
  |\sqrt{z} - I(z)| = \mathcal{O}(e^{-T}).
 \]
\end{lemma}

\begin{lemma}
 The quadrature error for $z = r \exp{(\pm \beta \pi i /2 )}$ with fixed $\beta \in [0,2)$ and $r \in [0,e^{4+2\beta-2T}]$ satisfies
 \[
  |I(z) - S(z)| = \mathcal{O}(e^{-T}).
 \]
\end{lemma}

\begin{lemma}\label{lem:poles2}
 For $r \in (0,1]$ and $z^+ = r\exp{(i\beta\pi/2)}$ (with fixed $\beta \in [0,2)$), as a function of $u$ the function $f(u,z^+)$ has poles at
\begin{small} \[
  \pi^+_{\pm,k} = \left(T + \frac12 \log r\right)^2 - \frac{\pi^2}{4} \left(\pm(2k+1) + \frac{\beta}{2}\right)^2 + i \pi \left(\pm (2k + 1) + \frac{\beta}{2} \right)  \left(T + \frac12\log r\right).
 \]\end{small}
 The residues of the poles $\pi^+_{\pm} = \pi_{\pm,0}^+$ closest to the real axis are $r^+_{\pm} = r_{\pm,0}^+ = \mp \frac{i}{\pi}\sqrt{z^+}$.
 For $r \in (0,1]$ and $z^- = r\exp{(-i\beta\pi/2)}$ (with fixed $\beta \in [0,2)$), as a function of $u$ the function $f(u,z^-)$ has poles at
 \begin{small} \[
  \pi^-_{\pm,k} = \left(T + \frac12 \log r\right)^2 - \frac{\pi^2}{4} \left(\pm(2k+1) - \frac{\beta}{2}\right)^2 + i \pi \left(\pm (2k + 1) -\frac{\beta}{2} \right)  \left(T + \frac12\log r\right).
 \]\end{small}
 The residues of the poles $\pi^-_{\pm} = \pi_{\pm,0}^-$ closest to the real axis are $r^-_{\pm} = r_{\pm,0}^- = \mp \frac{i}{\pi}\sqrt{z^-}$.
\end{lemma}
\begin{theorem}\label{thm:contour2}
 The quadrature error for $z = r \exp{(\pm \beta \pi i /2 )}$ with fixed $\beta \in [0,2)$ and $r \in [e^{4+2\beta-2T},1]$ is given by
 \[  I(z) - S(z) = \int_{\Gamma_1 \cup \Gamma_2} f(u,z) \kern2pt du + \int_\Gamma f(u,z) \delta(u) \kern2pt du - 2\pi i (r_{+} \delta(\pi_{+}) + r_- \delta(\pi_-)), \]
 in which $\Gamma_1 = [0,1-\beta/2]$, $\Gamma_2 = [4T^2+1-\beta/2,4T^2]$ and $\Gamma$ corresponds to the positively oriented rectangle $[1-\beta/2,4T^2+1-\beta/2] \times [-ai,ai]$, with $a = 2\pi (T + \frac12 \log r)$. Finally, $r_{\pm}$ and $\pi_{\pm}$ are the poles and residues of $f(u,x)$ according to Lemma~\ref{lem:poles2} (for $z = z^+ = r\exp{(i\beta\pi/2)}$ one has $\pi_\pm = \pi^+_\pm$, for $z = z^- = r\exp{(-i\beta\pi/2)}$ one has $\pi_\pm = \pi^-_\pm$).
\end{theorem}
\begin{theorem} \sloppy
 The truncation error $\sqrt{z}-I(z)$ and the size of the residues $r_{+} \delta(\pi_{+})$ and $r_{-} \delta(\pi_{-})$ in Theorem~\ref{thm:contour2} decay at the same rate in $N_t$, independently of $z = r e^{\pm i \beta \pi/2}$ with $r \in [0,1]$, with the choice
 \[
  h = (2-\beta)\pi^2
 \]
and that rate is $e^{-T} = e^{-\pi \sqrt{(2-\beta) N_t/4}}$.
\end{theorem}

To conclude the proof, the integrals in \cref{thm:contour2} remain to be bounded. It is easy to show that 
\[
\left| \int_{\Gamma_1 \cup \Gamma_2} f(u,x) \kern2pt du \right| = ±\mathcal{O}(e^{-T}).
\]
Furthermore, strong numerical evidence indicates that the remaining integral can be bounded as well.
\begin{conjecture} \sloppy
    For $z = r \exp{(\pm \beta \pi i /2 )}$ with $r \in [e^{4+2\beta-2T},1]$ and fixed $\beta~\in~[0,2)$
    \[
    \left| \int_\Gamma f(u,x) \delta(u) \kern2pt du \right| = \mathcal{O}(e^{-T})
    \]
    in which $\Gamma$ corresponds to the positively oriented rectangle $[1-\beta/2,4T^2+1-\beta/2] \times [-ai,ai]$, with $a = 2\pi (T + \frac12 \log r)$.
    \label{conj:bound2}
\end{conjecture}
Again, we numerically verified the size of the integral for both small and large values of $N_t$ and $r$ in $[0,1]$, as well as for different values of $\beta \in [0,2)$. A complete analytic understanding is even more involved than for~\eqref{conj:bound}.

\section*{Acknowledgments}
The authors wish to thank Nick Hale for a productive discussion.

\bibliographystyle{siamplain}
\bibliography{references}
\end{document}